\documentclass[12pt]{amsart}

\newif\ifanswers
\answerstrue

\usepackage{amsmath,amssymb}
\usepackage{graphicx}

\newtheorem{thm}{Theorem}[section]
\newtheorem{lem}[thm]{Lemma}
\newtheorem{prop}[thm]{Proposition}
\newtheorem{cor}[thm]{Corollary}
\newtheorem{rem}[thm]{Remark}
\newtheorem{notation}[thm]{Notation}
\newtheorem{definition}[thm]{Definition}

\newcommand{\Z}{{\mathbb Z}}

\newcommand{\R}{{\mathbb R}}
\newcommand{\T}{{\mathbb T}}
\newcommand{\sone}{{\mathcal{S}^{1}}}
\newcommand{\C}{\mathbb{C}}

\newcommand{\Pc}{\mathcal{P}}
\newcommand{\Ac}{\mathcal{A}}

\newcommand{\Sc}{\mathcal{S}}

\newcommand{\Fc}{\mathcal{F}}
\newcommand{\Mcc}{\mathcal{M}}
\newcommand{\Dc}{\mathcal{D}}
\newcommand{\TT}{\mathbb{T}}

\title[Lattice points on circles]{On probability measures arising from
  lattice points on circles}

\author{P\"ar Kurlberg}
\urladdr{www.math.kth.se/\~{ }kurlberg}
\address{Department of Mathematics, KTH Royal Institute of Technology,
SE-100 44 Stockholm, Sweden}
\email{kurlberg@math.kth.se}

\author{Igor Wigman}
\address{Department of Mathematics, King's College London, UK}
 \email{igor.wigman@kcl.ac.uk}

\begin{document}
\begin{abstract}
  A circle, centered at the origin and with radius chosen so that it
  has non-empty intersection with the integer lattice $\Z^{2}$,
  gives rise to a probability measure on the unit circle in a natural
  way.
  Such measures, and their weak limits, are said to be {\em attainable}
  from lattice points on circles.

  We investigate the set of attainable measures and show that it
  contains all extreme points, in the sense of convex geometry, of
  the set of all probability measures that are invariant under some
  natural symmetries.
  Further, the set of attainable measures is closed under convolution,
  yet there exist symmetric probability measures that are {\em not}
  attainable.
  To show this, we study the geometry of projections onto a finite
  number of Fourier coefficients and find that the set of attainable
  measures has many singularities with a ``fractal'' structure.  This
  complicated structure in some sense arises from prime powers ---
  singularities do not occur for circles of radius $\sqrt{n}$ if $n$
  is {\em square free}.


\end{abstract}

\date{\today}
\maketitle

\section{Introduction}
\label{sec:introduction}


Let $S$ be the set of nonzero integers expressible as a sum of two
integer squares. For $n\in S$, let
\begin{equation*}
\Lambda_{n} := \{ \vec{\lambda}= a+bi \in\Z[i]:\: a^{2}+b^{2}=n  \}
\end{equation*}
denote the intersection of the lattice $\Z[i] \subset \C$ with a
circle centered at the origin and of radius $\sqrt{n}$.
For $n \in S$, let $r_{2}(n) := |\Lambda_{n} |$ denote the cardinality
of $\Lambda_{n}$; for $n \not \in S$ it is convenient to define
$r_{2}(n)=0$.
We define a probability measure $\mu_{n}$ on the unit circle
$$
\mathcal{S}^{1} := \{ z \in \C : |z|=1 \}
$$
by letting
\begin{equation*}
\mu_{n} := \frac{1}{r_{2}(n)}\sum\limits_{\vec{\lambda}\in
  \Lambda_{n}}\delta_{\vec{\lambda}/\sqrt{n}},
\end{equation*}
where $\delta_{z}$ denotes the Dirac delta function with support at
$z$.  The measures $\mu_{n}$ are clearly invariant under
multiplication by $i$ and under complex conjugation.
We say that a
measure on $\mathcal{S}^{1}$ is {\em symmetric} if it is invariant
under these symmetries.

\begin{definition}
  A probability measure $\nu$ is said to be {\bf attainable from
    lattice points on circles}, or simply just {\bf attainable}, if
  $\nu$ is a weak limit point of the set $\{\mu_{n}\}_{n \in S}$.
\end{definition}
We note that any attainable measure is automatically symmetric.
Now, if two integers $m,n \in S$ are {co-prime},
\begin{equation}
\label{eq:mu(mn)=mu(m)*mu(n)}
\mu_{mn} = \mu_m \bigstar \mu_n,
\end{equation}
where $\bigstar$ denotes convolution on $\mathcal{S}^{1}$. Thus measures
$\mu_{n}$ for $n$ a prime power are of particular interest.  It
turns out that the closure of the set of measures given by $\mu_{p^e}$
for $p$ ranging over all primes $p \equiv 1 \mod 4$ and exponents $e$
ranging over integers $e \geq 1$ contains $\mu_{2^{k}}$, as well as
$\mu_{q^{2k}}$ for any prime $q \equiv 3 \mod 4$, and any exponent
$k \geq 0$.  (Note that $q^{l} \in S$ forces $l$ to be even.)

Motivated by the above,
we say that a measure $\mu$ is {\em prime power attainable} of $\mu$
is a weak limit point of the set $\{ \mu_{p^{e}} \}_{p \equiv 1 \text{ mod }
4, \, e \geq 1}$.  Similarly, we say that a measure $\mu$ is {\em
prime attainable} if $\mu$ is a weak limit point of the set $\{
\mu_{p} \}_{p \equiv 1 \text{ mod } 4} $.

\begin{prop}
\label{prop:A monoid}
The set of attainable measures is closed under convolution, and is
generated by the set of prime power attainable measures.

\end{prop}

Hence the set of attainable measures is the smallest closed w.r.t.
convolution set, containing all the prime power attainable measures.
The set of all symmetric probability measures is clearly a convex set,
hence equals the convex hull of its extreme points. Quite
interestingly, the set of prime attainable measures is exactly the set
of extreme points.
Now, since the set of attainable measures contains the extreme
points, and is closed under convolution one might wonder if {\em all}
symmetric probability measures are attainable?  By studying Fourier
coefficients of attainable measures we shall show that {\bf not all
symmetric measures are attainable}.

Given a measure $\mu$ on $\mathcal{S}$ and $k \in \Z$,
define the $k$-th Fourier coefficent of $\mu$ by
$$
\hat{\mu}(k) := \int_{\mathcal{S}} z^{-k} d\mu(z).
$$
If $\mu$ is symmetric it is straightforward to see that $\hat{\mu}(k)
= 0$ unless $4|k$.  Since $\mu$ is a probability measure,
$\hat{\mu}(0) = 1$, hence the first two informative Fourier
coefficients are $\hat{\mu}(4)$ and $\hat{\mu}(8)$; note that
$\hat{\mu}(-k) = \hat{\mu}(k)$ for all $k$ since $\mu$ is both real
and even (i.e. it is invariant under complex
conjugation).

\begin{thm}
\label{thm:muhat-region}
If $\mu$ is attainable and $|\hat{\mu}(4)|> 1/3$ then
\begin{equation}
\label{eq:mu-eight}
2\hat{\mu}(4)^{2}-1 \leq \hat{\mu}(8)
\leq \Mcc(\hat{\mu}(4)),
\end{equation}
where
\begin{equation}
\label{eq:max curve def}
\Mcc(x) = \max\left( x^{4}, (2|x|-1)^{2}   \right)
\end{equation}
denotes the ``max curve''.
Conversely, given $x,y$ such that $|x| \leq 1$ and $$2x^{2}-1 \leq y
\leq  \Mcc(x) ,$$ there exists an
attainable measure $\mu$ such that $( \hat{\mu}(4), \hat{\mu}(8)) = (x,y)$.
\end{thm}

For comparison, we  note that the
Fourier coefficients of the full set of symmetric
probability measures has the following quite simple description (see section \ref{sec:Fourier image}
below):
$$
\{ (\hat{\mu}(4), \hat{\mu}(8)) : \mu \text{ is symmetric} \}
= \{ (x,y) : |x| \leq 1, \ 2x^{2}-1 \leq y \leq 1 \}.
$$
As Figure~\ref{fig:all-symmetric} illustrates, the discrepancy between
all symmetric measures and the attainable ones is fairly large.  In
particular, note that the curves $y=x^{4}$, $y = 2x^{2}-1$, and
$(2|x|-1)^{2}$ all have the {\em same tangent} at the two points $(\pm
1, 1)$, consequently the set of attainable measures has cusps near
$(\pm 1, 1)$. However, there are attainable
measures corresponding to points {\em above the red curve} for $|x|
\leq 1/3$.
\begin{figure}[h]
  \centering
\ifanswers
\includegraphics[width=6cm]{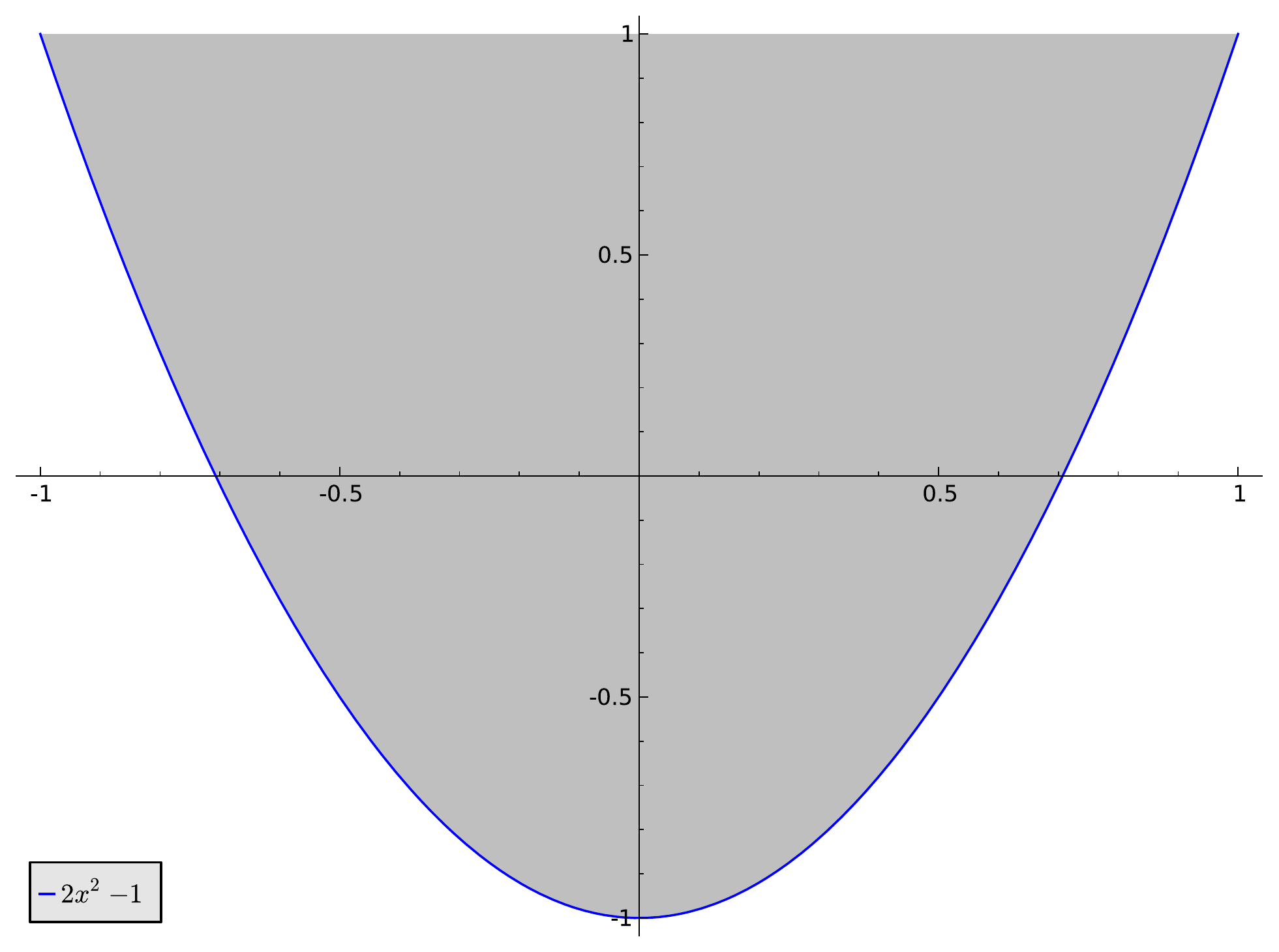}
\includegraphics[width=6cm]{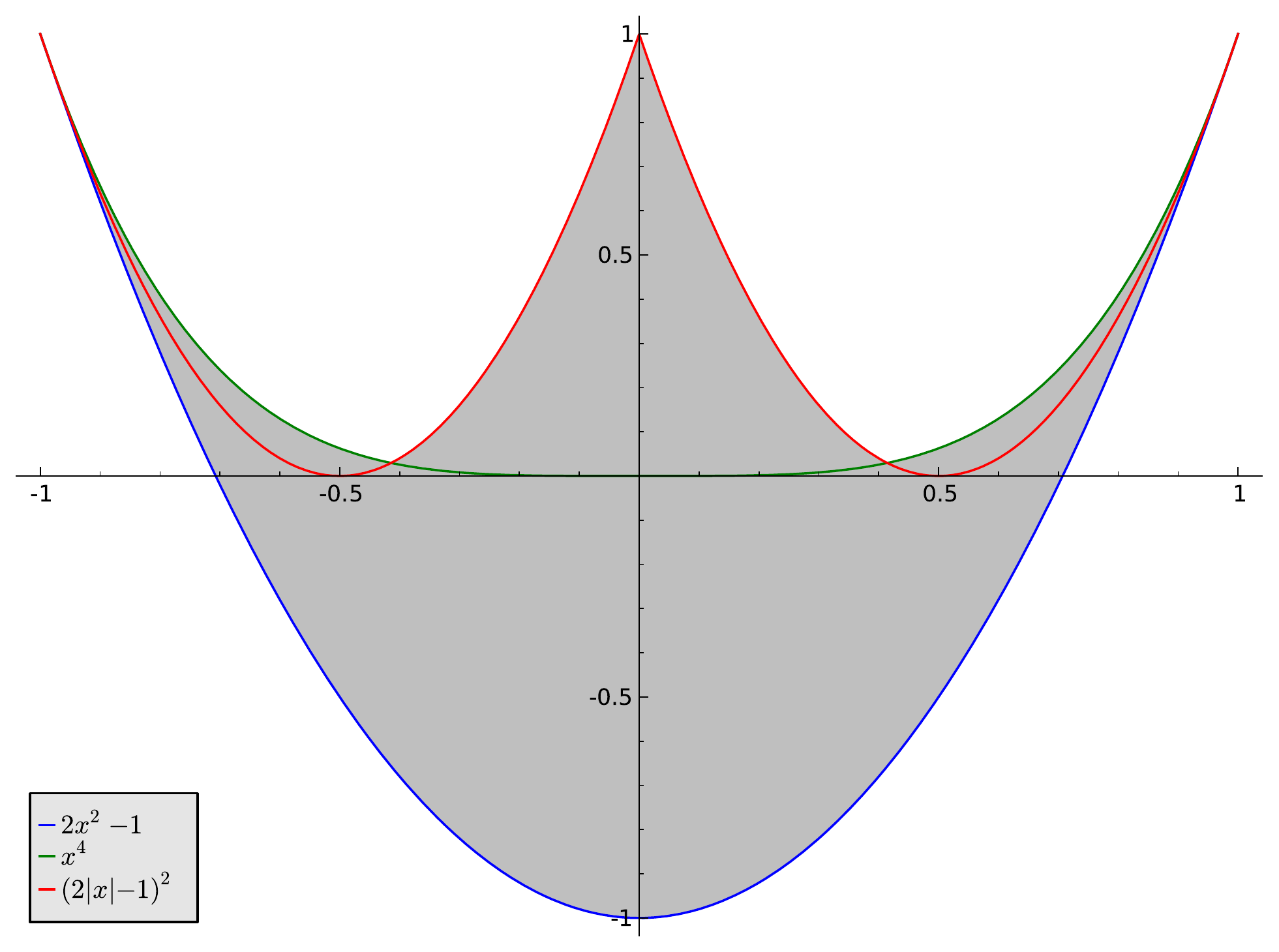}
\fi
\caption{Left: $\{ (\hat{\mu}(4), \hat{\mu}(8)) : \mu \text{ is
    symmetric}\}$.  Right: the region defined by the inequalities
  $2x^{2}-1 \leq y \leq \max \left(x^{4}, (2|x|-1)^{2} \right)$.  }
\label{fig:all-symmetric}
\end{figure}


\newpage{} To give an indication of the rate at which the admissible
region is ``filled out'', as well as illuminate what happens in the
region $|\hat{\mu}(4)| \leq 1/3$, we next present the results of some
numerical experiments in Figures~\ref{fig:int-attainable-up-to-1000}
and \ref{fig:int-attainable-up-to-10000}.
\begin{figure}[h]
\ifanswers
\includegraphics[width=6cm]{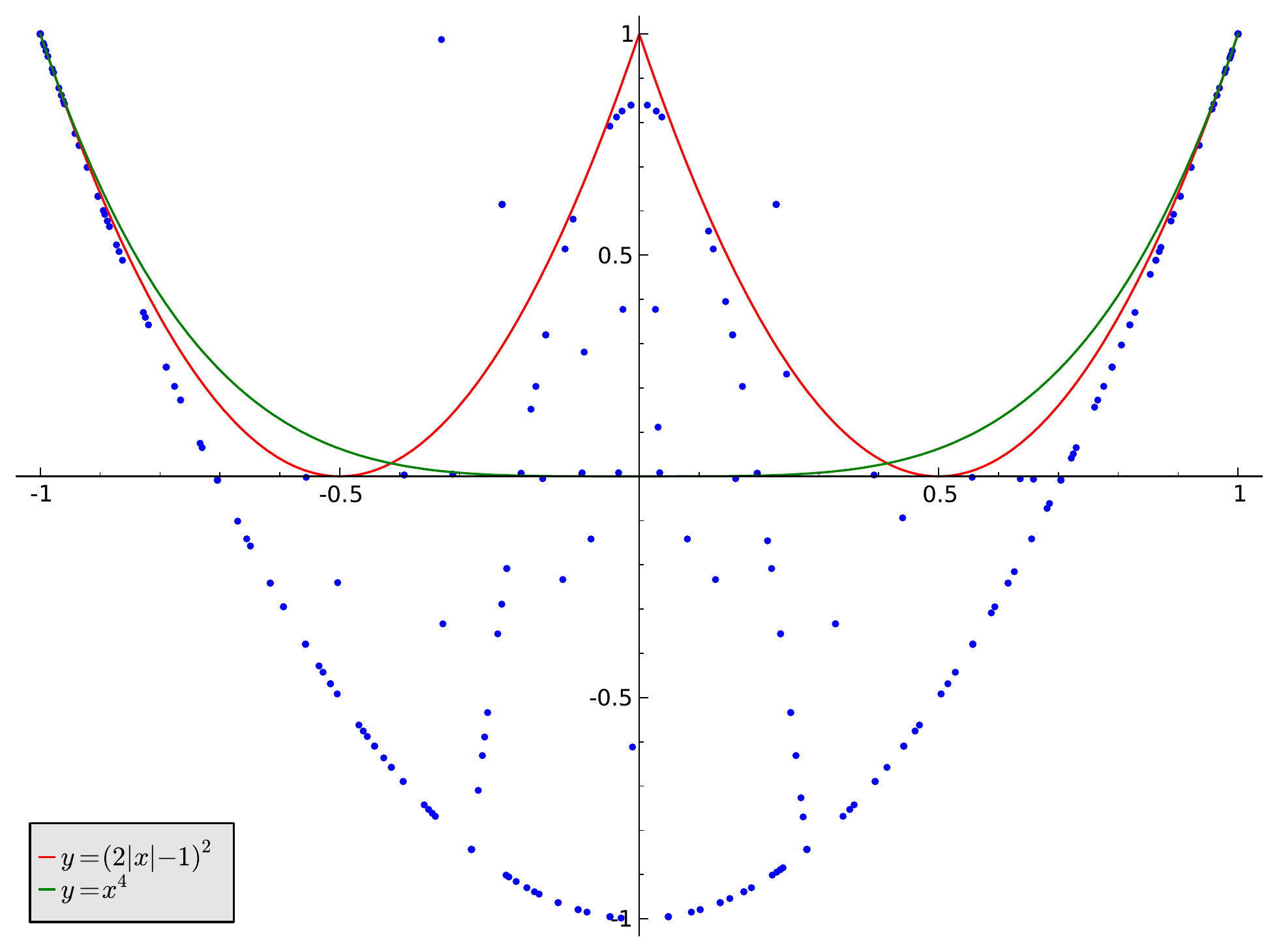}
\includegraphics[width=6cm]{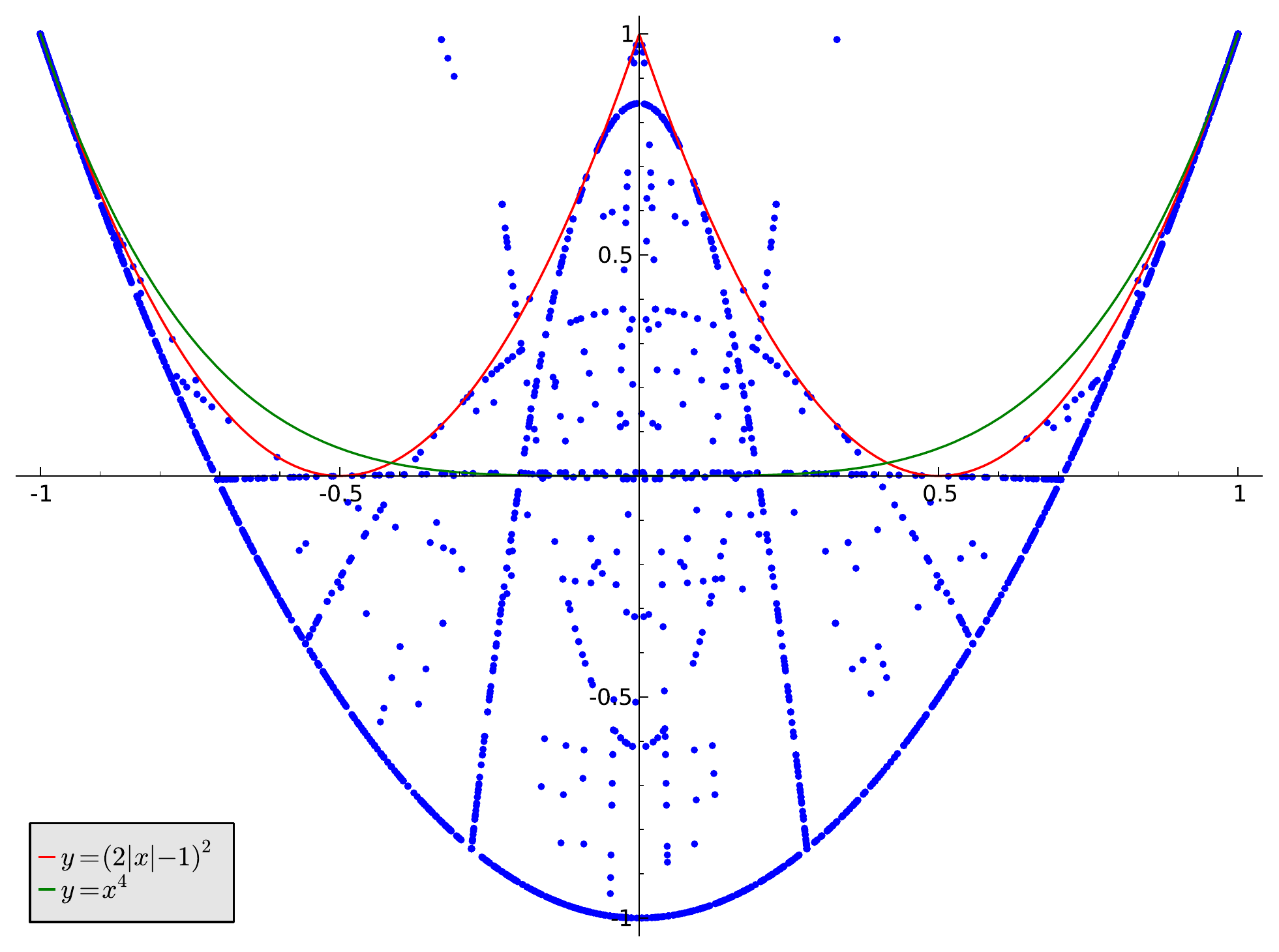}
\fi
\caption{Left: $(\hat{\mu_n}(4), \hat{\mu_n}(8))$ for $n\in S$,
$n\leq 1000$. Right: $(\hat{\mu_n}(4), \hat{\mu_n}(8))$ for $n \in S, n \leq 10000$.}
\label{fig:int-attainable-up-to-1000}
\end{figure}


Note that points lying clearly above the red curve, but below the
green one, are quite rare.  However,  ``spikes'' in the region
$|\hat{\mu}(n)| \leq 1/3$ are clearly present.
\begin{figure}[h]
\ifanswers
\includegraphics[width=6cm]{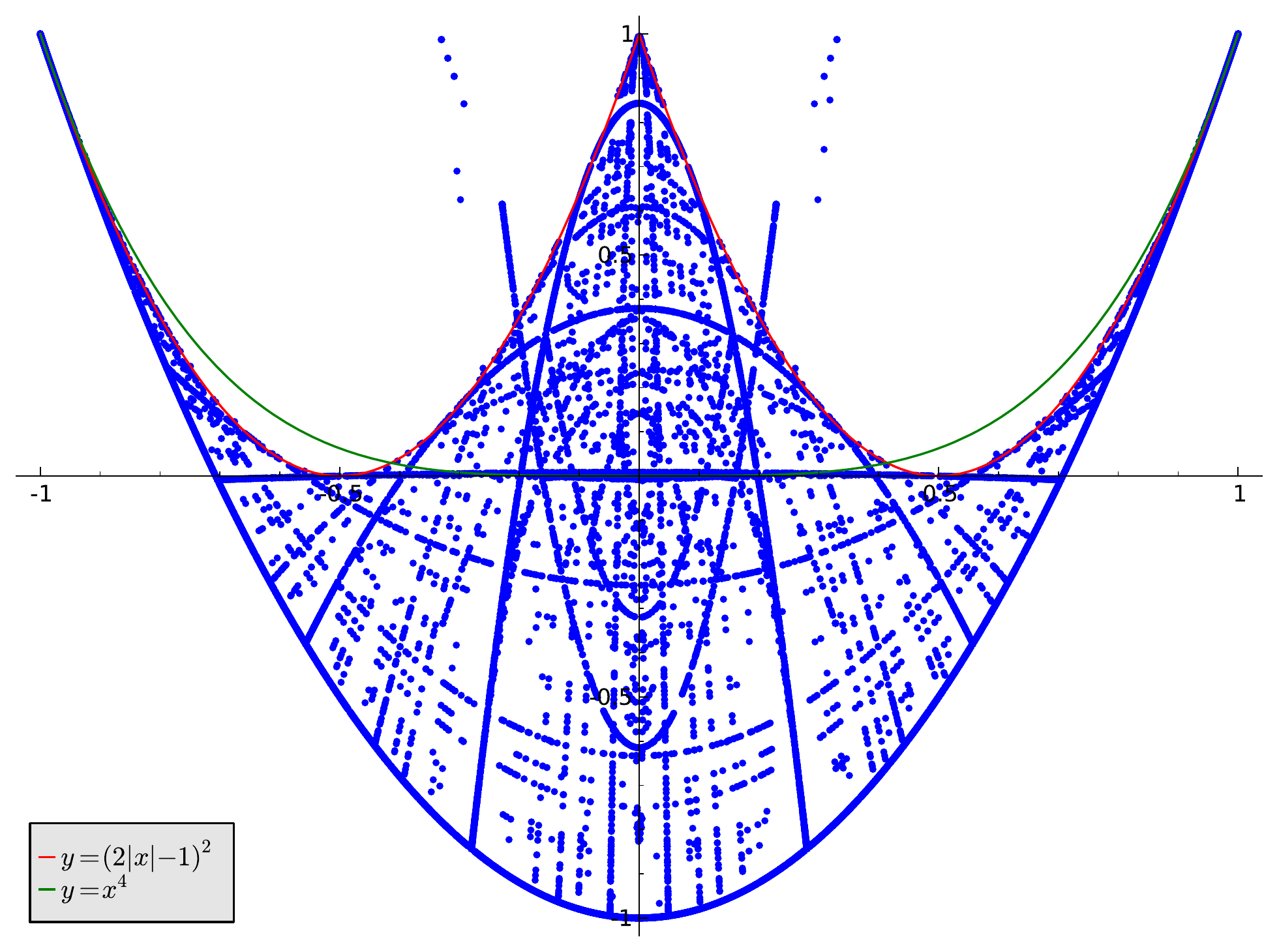}
\includegraphics[width=6cm]{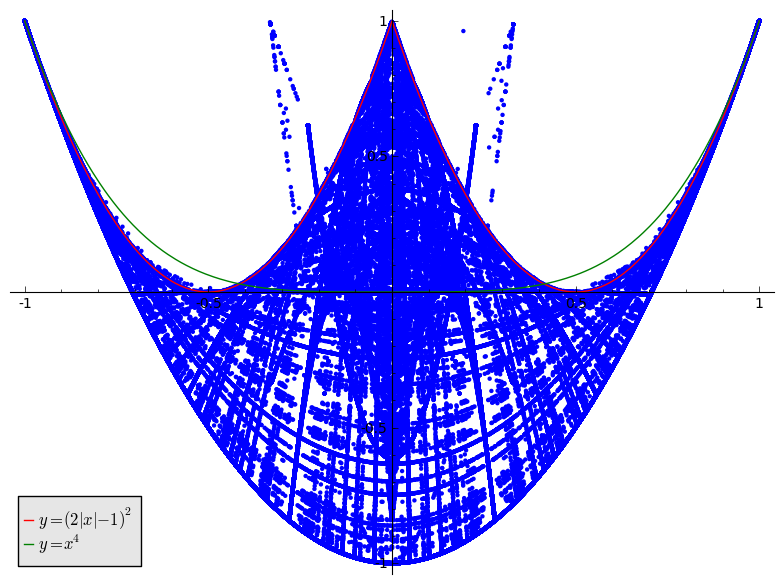}
\fi
\caption{Left: $(\hat{\mu_n}(4), \hat{\mu_n}(8))$
   for $n \in S, n \leq 100000$.
 Right: $(\hat{\mu_n}(4), \hat{\mu_n}(8))$ for $n \in S, n
 \leq 1000000$.}
\label{fig:int-attainable-up-to-10000}
\end{figure}


\subsection{Square free attainable measures}

As we shall see, the spikes in the region $|\hat{\mu}(4)| \leq
1/3$ are limits of measures $\mu_n$ where $n$ is divisible by $p^{e}$
for $e \geq 2$, but
for measures arising from square free $n \in S$, the structure is
much simpler.

We say that a measure $\mu$ is {\em square free attainable} if $\mu$
is a limit point of the set $\{ \mu_{n} : n \in S \text{ and $n$ is square free}\}$.
The set of square free attainable measures is also
closed under convolution, and it is easy to see that it is generated
by the set $\{\mu_p\}_{p \equiv 1 \mod 4}$, whose closure is the set of prime
attainable measures.

\begin{thm}
\label{thm:sqrfree attainable class}
If $\mu$ is square free attainable then
\begin{equation}
\label{eq:square free nec suff}
2\hat{\mu}(4)^{2}-1 \leq \hat{\mu}(8)
\leq \Mcc(\hat{\mu}(4)).
\end{equation}
Conversely, if $2x^{2}-1 \leq y \leq \Mcc(x)$ there
exists a square free attainable measure $\mu$ such that
$(\hat{\mu}(4), \hat{\mu}(8)) = (x,y)$.
\end{thm}

\subsection{Prime power attainable measures}
\label{sec:prime-power-atta}
As mentioned before, 
the spikes in the region $|\hat{\mu}(4)| \leq 1/3$ are due to measures
$\mu_n$ for which $n$ is divisible by a prime power $p^{e}$, for $e$
large.  Recall that a measure $\mu$ is prime power attainable if $\mu$
is a weak limit point of the set $\{\mu_{p^{e}}\}_{p \equiv 1 \mod 4,
e \geq 1}$.
If $\mu$ is a prime power attainable measure, then the point
$(\hat{\mu}(4), \hat{\mu}(8))$ can indeed lie {\em above} the curve
$\max( x^{4}, (2|x|-1)^{2})$ in the region $|\hat{\mu}(4)| \leq 1/3$, though
this phenomenon only occurs for even exponents (see
Figure~\ref{figg:prime-attainable}).
\begin{figure}[h]
\ifanswers
\includegraphics[width=5cm]{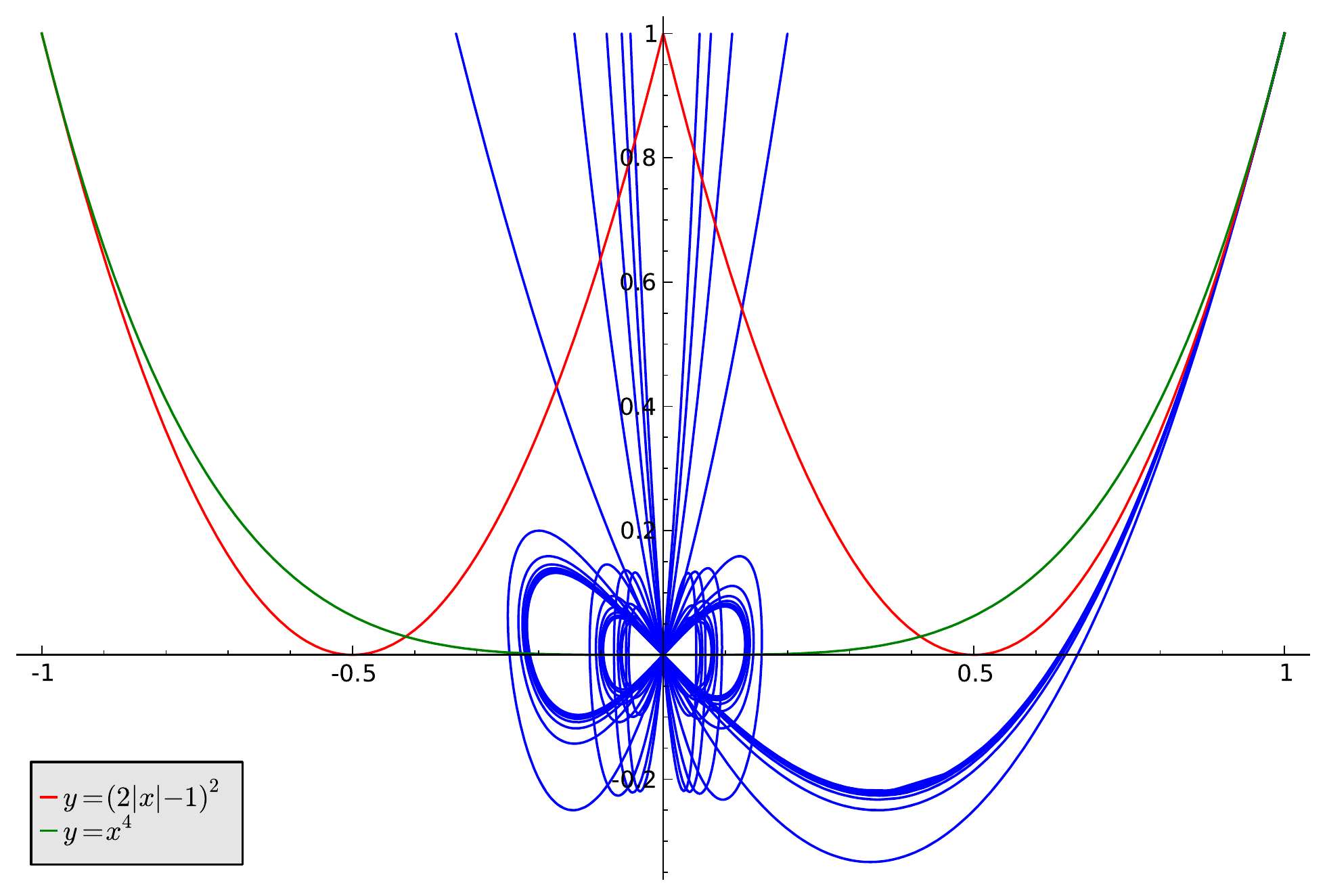}
\includegraphics[width=5cm]{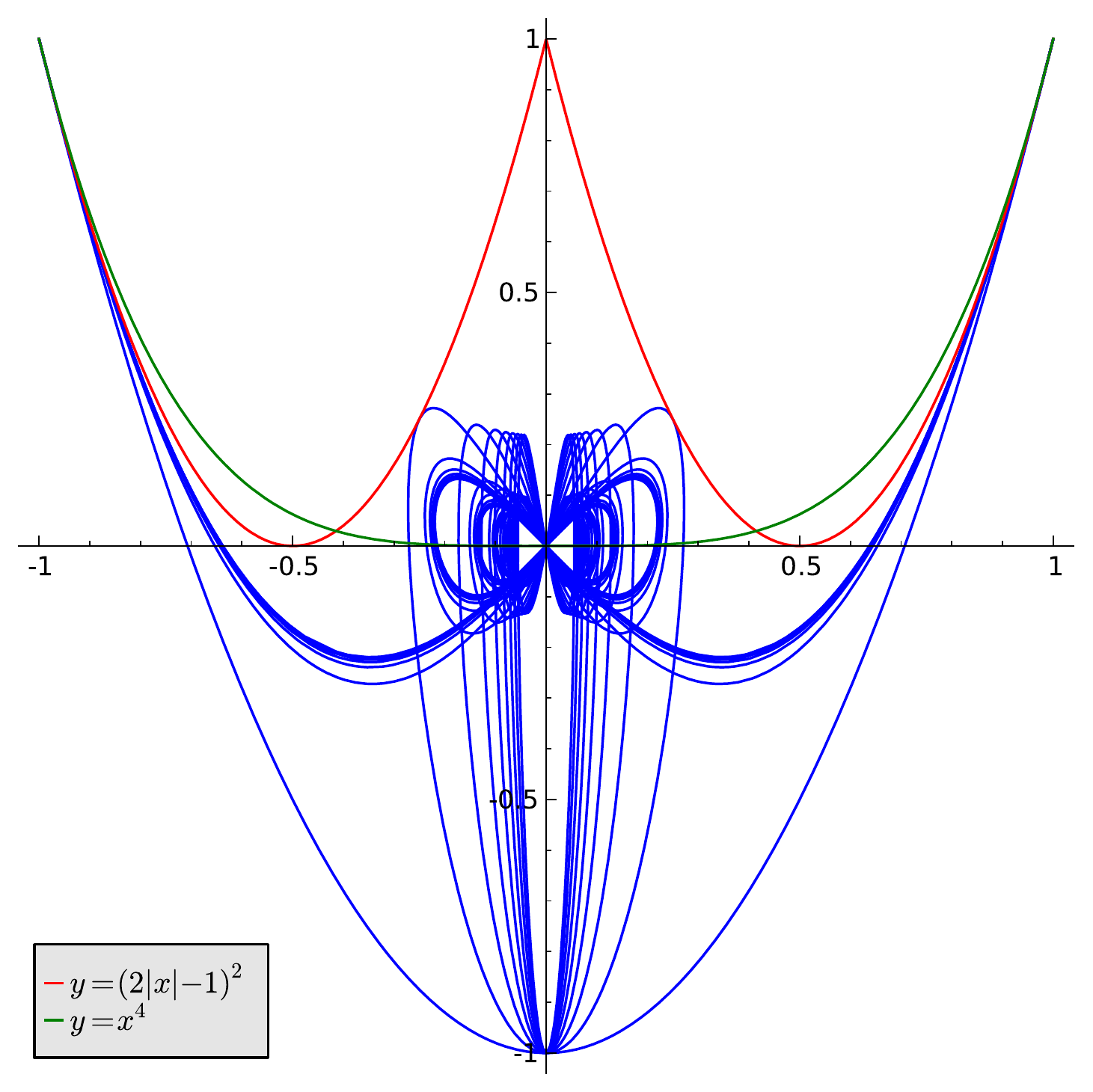}
\fi
 \caption{Prime power attainable measures attainable by $p^{M}$, $p\equiv 1(4)$ primes,
 $M\le 19$.  Left picture: even $M$. Right picture: odd $M$.}
\label{figg:prime-attainable}
\end{figure}
In fact, we will show that for every $k \in \Z^+$ there exists prime
power attainable $\mu$ such that $$(\hat{\mu}(4), \hat{\mu}(8)) = \left(
\frac{1}{2k+1}, 1\right).$$

\subsection{Fractal structure for $|\hat{\mu}(4)| \leq \frac{1}{3}$.}
\label{sec:thm:Ac2 top bar}

Let
\begin{equation}
\label{eq:Ac2 def}
\mathcal{A}_{2} := \{ (\hat{\mu}(4),  \hat{\mu}(8)) : \text{$\mu $ is
  attainable} \}
\end{equation}
denote the projection of the set of attainable measures onto the first
two non-trivial Fourier coefficients.  The intersection of
$\mathcal{A}_{2}$ with the vertical strip $\{ (x,y) : |x| \leq 1/3\}$
turns out to have a rather complicated fractal structure with
infinitely many spikes --- see Figure~\ref{fig:spikes}.  Since
$\Ac_{2}$ is closed under multiplication and $(-1,1)\in \Ac_{2}$ it
implies that it is invariant w.r.t.
\begin{equation}
\label{eq:(x,y)|->(-x,y)}
(x,y)\mapsto (-x,y),
\end{equation}
and hence we may assume $x\ge 0$.

\begin{figure}[h]
\ifanswers
\includegraphics[width=12cm]{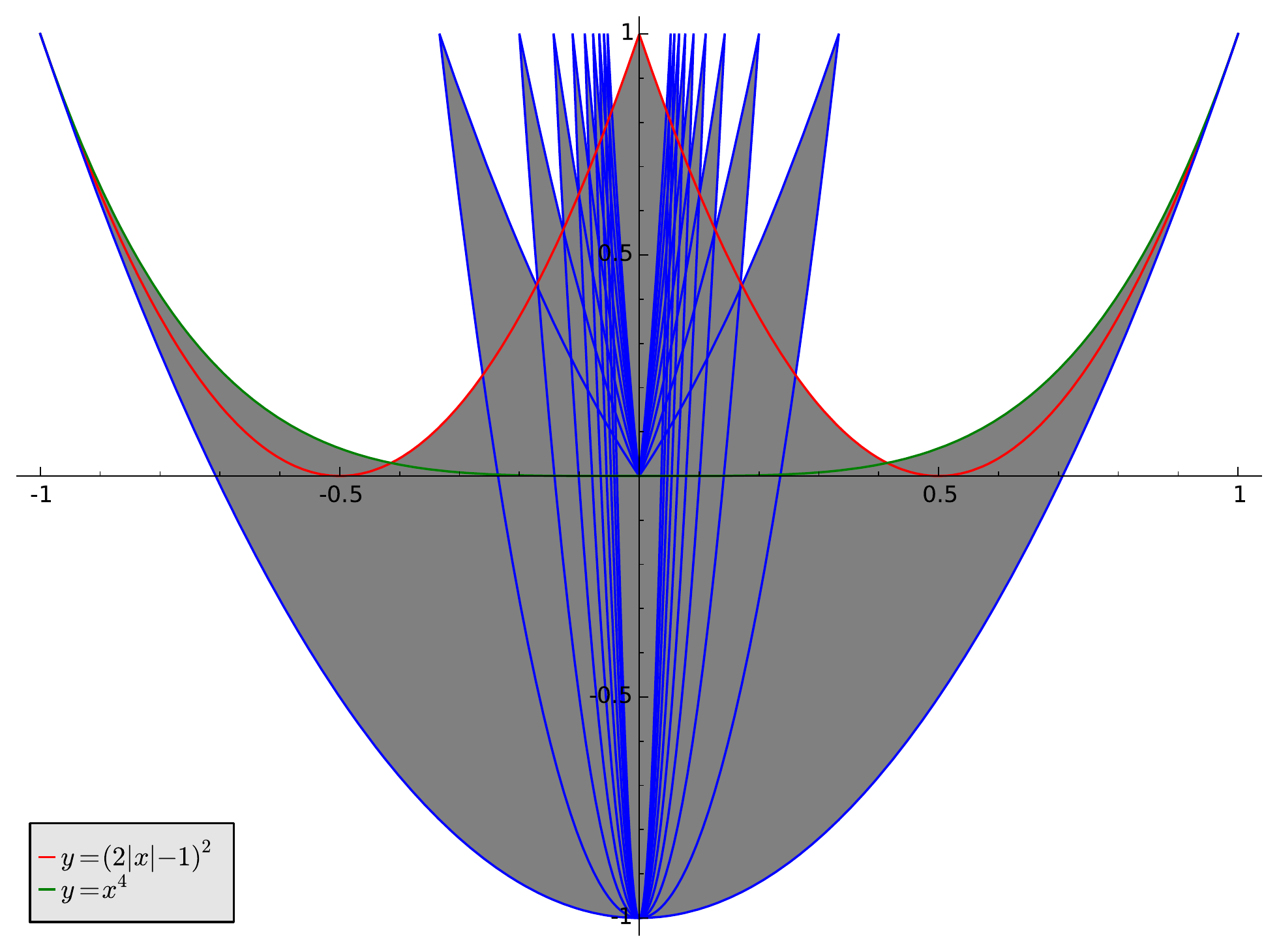}
\fi
\caption{Points $(\hat{\mu}(4), \hat{\mu}(8))$ for {some}
  attainable measures $\mu$ giving rise to spikes in the region
  $|\hat{\mu}(4)| \leq 1/3$.}
\label{fig:spikes}
\end{figure}

To be able to give a complete description of $\Ac_{2}$ we need a definition.

\begin{definition}

Let $x_{0}\in [0,1]$ and $a < x_{0}$.

\begin{enumerate}

\item We say that a pair of continuous functions
$$f_{1},f_{2}:(a,x_{0}]\rightarrow [0,1],$$ defines
a {\bf cornered domain between $a$ and $x_{0}$}
if for all $x\in (a,x_{0}]$ one has
$f_{1}(x)\le f_{2}(x)$, and $f_{1}(x)=f_{2}(x)$ if and only if
$x=x_{0}$, whence $f_{1}(x_{0})=f_{2}(x_{0})=1$.

\item For a pair of functions $f_{1},f_{2}$ as above the corresponding
{\bf cornered domain between $a$ and $x_{0}$} is
\begin{equation*}
\Dc_{a,x_{0}}(f_{1},f_{2}) = \{ (x,y)\in \R^{2}:\:  x\in (a,x_{0}],\, f_{1}(x) \le y\le f_{2}(x) \}.
\end{equation*}

\end{enumerate}

\end{definition}

The function $f_{1}$ and $f_{2}$ will be referred to as the ``lower
and upper" bounds for $\Dc_{a,x_{0}}(f_{1},f_{2})$
respectively.

\vspace{3mm}

\begin{thm}
\label{thm:Ac2 top bar}
\label{it:thm Ac2 top bar}
The intersection of the set $\Ac_{2}$ with the line $y=1$ equals
\begin{equation*}
\left\{ \left(\frac{\pm 1}{2k+1},1\right):\: k\ge 1 \right\}\cup
\{ (0,1)\} \cup \{ (\pm 1,1)\}.
\end{equation*}
Further, for $k\ge 1$, let $x_{k}=\frac{1}{2k+1}$ be the
$x$-coordinate of a point of the intersection described above. Then,
for every $k\ge 1$ there exists a pair of continuous piecewise
analytic functions $f_{1;k},\,f_{2;k}$ defining a cornered domain
between $0$ and $x_{k}$, so that $\Ac_{2}$ admits the following global
description:
\begin{multline}
\label{eq:Ac2=mix, spikes}
\Ac_{2}\cap \left\{ 0<x < \frac{1}{3} \right\} \\ =
\left(\bigcup\limits_{k=1}^{\infty}
  \Dc_{0,x_{k}}(f_{1;k},f_{2;k})\right)
\bigcup \left\{(x,y):0<x<\frac{1}{3},\,y\le (2x-1)^{2} \right\}.
\end{multline}

\end{thm}

Theorem \ref{thm:Ac2 top bar} is a rigorous
explanation of the thin strips or ``spikes" connecting all the
reciprocals of odd numbers on $y=1$, and the curve $y=(2|x|-1)^{2}$,
as in Figure \ref{fig:spikes}.  We remark that the functions $f_{1;k}$
and $f_{2;k}$ can with some effort be computed explicitly. The lower
bound $f_{1;k}$ is given as the (component-wise) product of
$(x_{k},1)$ by the parabola $y=2x^{2}-1$ mapping $(1,1)\mapsto
(x_{k},1)$; we re-parameterize the resulting curve $(x \cdot
x_{k},2x^{2}-1)$ so that it corresponds to the function
\begin{equation}
\label{eq:f1 def}
f_{1;k}(x) = \frac{2}{x_{k}^{2}} x^{2}-1,
\end{equation}
whose slope at $x_{k}$ is $f_{1}'(x_{k}) = 4(2k+1)$.

The upper bound
$f_{2}(x)$ is of a somewhat more complicated nature, see Definition \ref{def:f2 def}; it is analytic
around the corner with the slope $f_{2}'(x_{k}) = \frac{4}{3}(2k+1)$
(see the proof of Theorem \ref{thm:Ac2 top bar} in
section \ref{sec:fractal x>1/3}),
and it is {\em plausible} that it is (everywhere) analytic.
It then follows that the set
$\Ac_{2}$ has a discontinuity, or a jump, at $x=x_{k}$ (this is
a by-product of the fact that the slopes of both $f_{1}$ and $f_{2}$ at $x_{k}$
are {\em positive}.)

\subsection{Discussion}
Our interest in attainable measures originates in the study \cite{KKW}
of zero sets (``nodal lines'') of random Laplace eigenfunctions on the
standard torus $\T := \R^2/\Z^2$.
%
More precisely, for each $n \in S$ there is an associated Laplace
eigenvalue given by $4\pi^2 n$, with eigenspace dimension equal to
$r_{2}(n)$.  On each such eigenspace there is a natural notion of a
``random eigenfunction'', and the variance (apropriately normalized)
of the nodal line lengths of these random eigenfunctions equals
$(1+\widehat{\mu_{n}}(4)^{2})/512 + o(1)$ as $r_{2}(n) \to \infty$.
It was thus of particular interest to show that the accumulation
points of $\widehat{\mu_{n}}(4)^{2}$, as $n \in S$ tends to infinity
in such a way that also the eigenspace dimension $r_{2}(n) \to \infty$, is
maximal --- namely the full interval $[0,1]$.  This is indeed the case
(cf. \cite[Section~1.4]{KKW}), but a very natural question is: which
measures are attainable?

In order to obtain asymptotics for the above variance it is essential
to assume that the eigenspace dimension grows, and one might wonder if
``fewer'' measures are attainable under this additional assumption.
However, as the following shows, this is not the case (the proof can
be found in section \ref{sec:attain<=>strongly}.)
\begin{prop}
\label{prop:attain<=>strongly}
A measure $\mu\in\mathcal{P}$ is attainable (i.e. $\mu\in\Ac$), if and
only if there exists a sequence $\{n_{j}\}$ such that
$\mu_{n_{j}}\Rightarrow \mu$ with the additional property that
$r_{2}(n_{j})\rightarrow\infty$.
\end{prop}

\subsection{Outline}
\label{sec:outline}
For the convenience of the reader we briefly outline the contents of
the paper. In Section \ref{sec:exmp attain} we give some explicit examples of attainable,
and non-attainable measures, and describe our motivation for studying
the set of attainable measures.  In Section \ref{sec:Fourier} we give a brief
background on Fourier coefficients of probability measures, and in
Section \ref{sec:NT background} we recall some needed facts from number theory
along with proving the more basic results above.  Section \ref{sec:muhat-region}
contains the proof of Theorem \ref{thm:muhat-region} (a complete
classification of attainable measures in the region
$|\hat{\mu}(4)|>1/3$), and Section \ref{sec:fractal x>1/3} contains the proof of
Theorem \ref{thm:Ac2 top bar} (the complete classification of attainable
measures in the region $|\hat{\mu}(4)| \leq 1/3$), postponing
some required results of technical nature to the appendix.  Finally, in
Section \ref{sec:sqrfree attainable class}, we classify the set of square-free attainable measures.

\subsection{Acknowledgements}


We would like to thank Ze\'{e}v Rudnick and Mikhail Sodin for raising
the problem considered in this manuscript, and the many fruitful
discussions concerning various subjects related to the presented
research. We thanks Fedor Nazarov and Peter Sarnak for many
stimulating and fruitful discussions leading to
some improvements of our results.

P.K. was partially supported by grants from the G\"oran
  Gustafsson Foundation for Research in Natural Sciences and
Medicine, and the Swedish Research Council (621-2011-5498).
The research leading to these results has received funding from the
European Research Council under the European Union's Seventh
Framework Programme (FP7/2007-2013), ERC grant agreement
n$^{\text{o}}$ 335141 (I.W.), and
an EPSRC Grant EP/J004529/1 under the First Grant Scheme (I.W.).

\section{Examples of attainable and unattainable measures}
\label{sec:exmp attain}

\subsection{Some conventions}

Let
$$\tilde{\delta}_{0} := \frac{1}{4}\sum_{k=0}^{3}\delta_{i^{k}} $$
be the atomic probability measure supported at the $4$ symmetric
points $\pm 1$, $\pm i$ (``Cilleruelo measure"). Given an angle $\theta \in [0,\pi/4]$, let
\begin{equation}
\label{eq:tilde delta theta def}
\tilde{\delta}_{\theta} :=
\tilde{\delta}_{0} \bigstar ( \delta_{e^{i \theta}} + \delta_{e^{-i
    \theta}} )/2
= \frac{1}{8}
\sum_{k=0}^{3} \left(
\delta_{e^{i (\pi k/2+\theta)}} + \delta_{e^{i (\pi k/2-\theta)}}
\right);
\end{equation}
recall that $\bigstar$ denotes convolution on $\mathcal{S}^{1}$.
For $\theta=0,\pi/4$ the measure $\tilde{\delta}_{\theta}$ is supported
at $4$ points whereas for all other values of $\theta$ the support consists
of $8$ points.
Given an integer $m \geq 1$ and $\theta \in [0,\pi/4]$, let
$$
\tilde{\delta}_{\theta,m} :=
\tilde{\delta}_{0} \bigstar
\left(
 \frac{1}{m+1} \sum_{j=0}^{l}\delta_{e^{i \theta(m-2j)} } \right).
$$
We note that $\tilde{\delta}_{\theta} = \tilde{\delta}_{\theta,1}$,
and that $\mu$ is symmetric if and only if $\mu$ is invariant under
convolution with $\tilde{\delta}_{0}$; convolving with
$\tilde{\delta}_{0}$ is a convenient way to ensure that a measure is
symmetric.

\subsection{Some examples of attainable and unattainable measures}

Given $\theta \in [0,\pi/4]$ let $\tau_{\theta}$ denote the symmetric
probability measure with uniform distribution on the four arcs given
by
$$\{ z : |z|=1, \arg(z) \in \cup_{k=0}^{4} [k\pi/2-\theta,k\pi/2+\theta
]\}.$$ Using some well known number theory given below (cf. section
\ref{sec:NT background}) it is straightforward to show that
$\tau_{\theta}$ is attainable for all $\theta \in [0,\pi/4]$.  In
particular, $d\mu_{\text{Haar}} = d\tau_{\pi/4}$, the Haar measure on
$\mathcal{S}^{1}$ normalized to be a probability measure, is
attainable.    In fact,
it is well known (see e.g. ~\cite{FKW}) that there exists a density
one subsequence $\{n_{j}\}\subseteq S$, for which the corresponding
lattice points $\Lambda_{n_{j}}$ become equidistributed on the circle;
this gives another construction of $d\mu_{\text{Haar}}$ as an
attainable measure.


It is also possible to construct other singular measures.  In
Section~\ref{sec:NT background} we will outline a construction of
attainable measures, uniformly supported on Cantor sets.
Moreover, if $q$ is a prime congruent to $3$ modulo $4$ it is well
known that the solutions to $a^{2}+b^{2} = q^{2}$ are given by $(a,b)
= (0,\pm q)$, or $(\pm q, 0)$, thus $\tilde{\delta}_{0}$ is
attainable.  A subtler fact, due to Cilleruelo, is that there exists
sequences $\{n_{j}\}_{j\ge 1}$ for which $\Lambda_{n_{j}}$ has very
singular angular distribution even though the number of points
$r_{2}(n_{j})$ tends to infinity.  Namely, it is possible to force all
angles to be arbitrarily close to integer multiples of $\pi/2$, hence
$\frac{1}{4}\sum_{k=0}^{3}\delta_{i^{k}}$ is an accumulation point of
$d\mu_{n_{j}}$ as $n_{j} \to \infty$ in such a way that $r_{2}(n_{j})
\to \infty$.

We may also construct some explicit {\em unattainable} probability
measures on $\Sc^{1}$ satisfying all the symmetries; in fact the
following corollary of Theorem \ref{thm:Ac2 top bar} constructs
explicit unattainable measures, remarkably supported on $8$ points
only --- the minimum possible for symmetric unattainable
measures.

\begin{cor}[Corollary from Theorem \ref{thm:Ac2 top bar}]
\label{eq:upsa unattainable}
The probability measure $$\eta_{a}:=a\tilde{\delta}_{0}+(1-a)\tilde{\delta}_{\pi/4}$$ is
attainable, if and only if $a= 0,\,\frac{1}{2},\,1$ or $a$ is
of the form
$$a=\frac{1}{2}\pm\frac{1}{2(2k+1)}$$ for some $k\ge 1$.
\end{cor}

\section{Fourier analysis of probability measures}

\label{sec:Fourier}

\subsection{Some notation and de-symmetrization of probability measures.}

It is convenient to work with two models: either with the unit circle
embedded in $\C$, or
$$
\mathbb{T}^{1}:= \R/2\pi\Z.
$$
Rather than working with $\{ \mu_{n} \}$ and its weak partial limits,
for notational convenience we work with their de-symmetrized variants,
i.e.
\begin{equation}
\label{eq:nun desymmetrized mun}
d\nu_{n}(\theta) = d\mu_{n}\left(\frac{\theta}{4}\right),
\end{equation}
$\theta\in \mathbb{T}^{1}$. The measures $\nu_{n}$ are invariant under
complex conjugation (where thought of $\mathcal{S}^{1}\subseteq \C$);
equivalently,
for $\theta\in \mathbb{T}^{1}$,
\begin{equation*}
d\nu_{n}(-\theta)=d\nu_{n}(\theta).
\end{equation*}

\begin{notation}
\label{not:P,A def}

Let $\mathcal{P}$ be the set of all probability measures $\mu$ on
$\mathcal{S}^{1}$ satisfying for $\theta\in \mathbb{T}^{1}$
\begin{equation}
\label{eq:mu(theta)=mu(-theta)}
d\mu(-\theta) = d\mu(\theta).
\end{equation}

Further, let $\mathcal{A}\subseteq\mathcal{P}$ be the set of
all weak partial limits of $\{\nu_{n}\}$
i.e. all probability measures $\mu\in\mathcal{P}$ such that there exists
a sequence $\{ n_{j}\}$ with $$\nu_{n_{j}}\Rightarrow \mu.$$


\end{notation}

The set $\Ac$ defined above is the de-symmetrization
of the collection of attainable measures via \eqref{eq:nun desymmetrized mun}; by abuse
of notation we will refer to
the elements of $\Ac$ as attainable measures.
One may restate Proposition \ref{prop:A monoid} as stating that $\Ac$ is closed w.r.t. convolutions;
thus $\Ac$ is an abelian monoid with identity $\delta_{0}\in \Ac$.
The effect of the de-symmetrization \eqref{eq:nun desymmetrized mun} is that for all $m\in \Z$
$$\widehat{\nu_{n}}(m) = \widehat{\mu_{n}}(4m);$$ since by the $\pi/2$-rotation invariance of $\mu_{n}$,
$\widehat{\mu}(k) = 0$ unless $k$ is divisible by $4$, this transformation preserves all the information.

\subsection{Measure classification on the Fourier side}

\label{sec:Fourier image}

We would like to study the image of $\Ac$ under Fourier transform, or,
rather, its projections into finite dimensional spaces. Since
$\Ac\subseteq \Pc$ we first study the Fourier
image of the latter; a proper inclusion of the image of $\Ac$ inside
the image of $\Pc$ would automatically imply the existence of
unattainable measures $\mu\in\Pc\setminus\Ac$.

For $\theta\in (0,\pi)$ let $\upsilon_{\theta}$ be the probability measure
\begin{equation}
\label{eq:upsilontheta def}
\upsilon_{\theta} = \frac{1}{2}\left(\delta_{\theta}+\delta_{-\theta}\right),
\end{equation}
and for the limiting values $\theta=0,\pi$ we denote
$\upsilon_{0}=\delta_{0}$ and $\upsilon_{\pi} = \delta_{\pi}$. As for
$\theta\in [0,\pi]$, $\delta_{\theta}$ are the de-symmetrizations of
$\tilde{\delta}_{\theta/4}$ in \eqref{eq:tilde delta theta def}, and
it then follows that $\upsilon_{\theta}\in \Ac$. Clearly
(see e.g. ~\cite{Krein-Nudelman}, Chapter $1$) the set $\Pc$ is the convex hull
of $$\{\upsilon_{\theta}:\: \theta\in [0,\pi]\}.$$

\vspace{3mm}

Let $\Pc_{k}\subseteq\R^{k}$ be the image of $\Pc$ under the
projection $\Fc_{k} : \Pc \to R^{k}$ given by
$$\Fc_{k}(\mu) :=(\widehat{\mu}(1),\ldots \widehat{\mu}(k)),$$
i.e. $\Pc_{k}=\Fc_{k}(\Pc_{k})$
are the first $k$ Fourier coefficients of the measure $\mu$ as $\mu$
varies in $\Pc$. Recalling the invariance
\eqref{eq:mu(theta)=mu(-theta)} for $\mu\in\Pc$ we may write
\begin{equation*}
\Fc_{k}\mu = (\widehat{\mu}(1),\ldots \widehat{\mu}(k)) = \int\limits_{\Sc^{1}}\gamma_{k}(\theta)d\mu(\theta),
\end{equation*}
where $\gamma_{k}$ is the curve
\begin{equation*}
\gamma_{k}(\theta) = (\cos(\theta),\cos(2\theta),\ldots,\cos(k\theta))
\end{equation*}
$\theta\in [0,\pi]$. Thus $\Pc_{k}=\Fc_{k}(\Pc)$ could be regarded as a convex combination of points lying on $\gamma_{k}$
(corresponding to $\upsilon_{\theta}$);
it would be then reasonable to expect $\Pc_{k}$ to be equal to the convex hull of $\gamma_{k}$.

This intuition was made rigorous in a more general scenario by
F. Riesz ~\cite{Riesz 1911} in a classical theorem on the generalized
moments problem (cf. ~\cite{Krein-Nudelman}, Chapter 1, Theorem $3.5$
on p. 16).  The sets $\Pc_{k}$ are the convex hulls of the curves
$\gamma_{k}$ in $\R^{k}$ indeed. Interestingly, since
$\cos(m\theta)$ is a polynomial in $\cos(\theta)$, the curve $\gamma_{k}$ is
algebraic. As a concrete
example, for $k=2$ the image $\Pc_{2}$ of $\Pc$ under
$$\Fc_{2}:\mu\mapsto (\widehat{\mu}(1),\widehat{\mu}(2))$$ is the convex hull of the parabola $y=2x^{2}-1$,
$x\in [-1,1]$, i.e. the set
\begin{equation}
\label{eq:P2 Riesz}
\Pc_{2} = \{(x,y): x\in [-1,1],\,2x^{2}-1\le y \le 1 \} ,
\end{equation}
as shown in Figure \ref{fig:all-symmetric}, to the left.

Analogously to the above,
define $$\Ac_{k}=\Fc_{k}(\Ac)\subseteq\Pc_{k},$$
(cf. \eqref{eq:Ac2 def}, and bear in mind the de-symmetrization
\eqref{eq:nun desymmetrized mun}).
Since, by the definition, $\Ac$ is closed in $\Pc$ (i.e. the weak limit set of $\Ac$ satisfies
$\Ac'\subseteq \Ac$), if follows that for every $k\ge 2$, $\Ac_{k}$ is closed in $\Pc_{k}$ in the
usual sense. The shell $y=2x^{2}-1$ of the convex hull $\Pc_{2}$ is
(uniquely) attained by the family $\{\upsilon_{\theta}\ :\theta\in
[0,\pi]\}$ of measures as in \eqref{eq:upsilontheta def} with the
Fourier coefficients
\begin{equation}
\label{eq:upsilon_theta Fourier}
(\widehat{\upsilon_{\theta}}(1),\widehat{\upsilon_{\theta}}(2))=(\cos(\theta),\cos(2\theta)).
\end{equation}

Finally, it is worth mentioning that the set $\Ac$ is {\em not convex}, as $\Ac_{2}$
contains the parabola
$$\{(x,2x^{2}-1):\: x\in [-1,1]\} \subseteq \Ac_{2},$$ whose points
correspond to the measures \eqref{eq:upsilontheta def}, though not its
convex hull.  
(In other words, had $\Ac$ been convex, that would force all symmetric measures to be attainable.)

\section{Proofs of the basic results}

\label{sec:NT background}

\subsection{Number theoretic background}

We start by giving a brief summary on the structure of $\Lambda_{n}$
(equivalently, $\mu_{n}$ or their de-symmetrized by \eqref{eq:nun
  desymmetrized mun} versions $\nu_{n}$) given the prime decomposition
of $n$. These results follow from the (unique) prime factorization of Gaussian
integers, see e.g. ~\cite{Cil}. First, for every ``split'' prime $$p\equiv 1
\mod{4},$$ there exists an angle $\theta_{p}\in [0,\pi]$, such that
the measure $\nu_{p}$ arising from $p$ is given by
$$
\nu_{p} = \upsilon_{\theta_{p}}=(\delta_{\theta_p} + \delta_{-\theta_p} )/2.
$$
More generally, if a split prime $p$ occurs to a power $p^{e}$, we
find that the resulting measure is given by
$$
\nu_{p^e}
= \upsilon_{\theta_{p},e},
$$
where
\begin{equation}
\label{eq:upsilonM,theta def}
\upsilon_{\theta;M}= \frac{1}{M+1}\sum\limits_{k=0}^{M}\delta_{(M-2k)\theta},
\end{equation}
and hence, in particular, $$r_{2}(p^{e}) = 4(e+1)$$ (recall the de-symmetrization \eqref{eq:nun desymmetrized mun}).
Both the $\{\nu_{n}\}$ and $\frac{1}{4}r_{2}(n)$ are multiplicative in the sense that for $n_{1},n_{2}$ co-prime numbers $(n_{1},n_{2})=1$,
\begin{equation}
\label{eq:nu_n1n2=nu_n1*nu_n2}
\nu_{n_{1}\cdot n_{2}} = \nu_{n_{1}}\bigstar \nu_{n_{2}},
\end{equation}
and
$$r_{2}(n_{1})r_{2}(n_{2})=4r_{2}(n_{1} n_{2}).$$
In particular, $r_{2}(n)=0$ unless $n$ is of the form
\begin{equation*}
n=2^{a}p_{1}^{e_{1}}\cdot \ldots\cdot p_{k}^{e_{k}}q_{1}^{2r_{1}}\cdot\ldots\cdot q_{l}^{2r_{l}},
\end{equation*}
for $p_{i}\equiv 1\mod 4$, $q_{j}\equiv 3\mod 4$ primes (in
particular, all the exponents of primes $\equiv 3 \mod 4$ are even); in
this case
\begin{equation*}
\nu_{n} = \bigstar_{i=1}^{k} \nu_{p_{i}^{e_{i}}},
\end{equation*}
and $$r_{2}(n) = 4\prod\limits_{i=1}^{k}(e_{i}+1).$$

By Hecke's  celebrated result
~\cite{hecke20-eine-neue-art-I,hecke20-eine-neue-art-II} the angles
$\theta_{p}$ are equidistributed in $[0,\pi/4]$: for
every $0 \le \alpha < \beta \le \pi$,
\begin{equation*}
\#\{ p \le X, \, p\equiv 1(4): \: \theta_{p}\in [\alpha,\beta] \} \sim
\frac{(\beta-\alpha)}{\pi/4} \cdot \frac{X}{2\log{X}}
\end{equation*}
In particular, the following lemma is an immediate consequence.
\begin{lem}
\label{lem:|thetap-theta|<eps}
For every $\theta\in [0,\pi]$ and $\epsilon>0$ there exist a split prime $p$ with
\begin{equation*}
|\theta_{p}-\theta|<\epsilon.
\end{equation*}
\end{lem}

\subsection{Proof of Proposition \ref{prop:A monoid}}

\begin{proof}

We will prove the equivalent de-symmetrized version of the statement, i.e.
that if $\gamma_{1},\gamma_{2}\in\Ac$ then $$\gamma_{1}\bigstar\gamma_{2}\in\Ac.$$
Let $\{m_{k}\},\{n_{k} \}\subseteq S$ be two sequences so that $\nu_{m_{k}}\Rightarrow\gamma_{1}$,
$\nu_{n_{k}}\Rightarrow\gamma_{2}$.
We would like to invoke the multiplicativity
\eqref{eq:nu_n1n2=nu_n1*nu_n2} of $\{ \nu_{n} \}$; we cannot apply it
directly, as $n_{k}$ and $m_{k}$ may fail to be co-prime. To this end
rather than using $\nu_{m_{k}}$ we are going to
substitute\footnote{One may think about this procedure as a number theoretical
  analogue of choosing an independent identically distributed copy of
  a given random variable.} it with
$\nu_{m'_{k}}$ chosen to approximate $\nu_{m_{k}}$, so that $m'_{k}$ is co-prime
to $m_{k}$, via Lemma \ref{lem:|thetap-theta|<eps}. In the remaining part of the proof we shall argue that
\begin{equation}
\label{eq:nunk*m'k=nunk*num'k}
\nu_{n_{k}\cdot m'_{k}} = \nu_{n_{k}}\bigstar \nu_{m'_{k}} \Rightarrow \gamma_{1}\bigstar \gamma_{2},
\end{equation}
provided we care to choose $m'_{k}$ so that $\nu_{m'_{k}}$ approximates $\nu_{m_{k}}$
sufficiently well.

To this end it is more convenient to work with the space of Fourier coefficients;
the weak convergence of probability measures corresponds to
point-wise convergence of the Fourier coefficients. By Lemma
\ref{lem:|thetap-theta|<eps} we may replace $m_{k}$ with $m'_{k}$
co-prime to $n_{k}$ that satisfies for every $j \le k$
\begin{equation*}
\left | \int \chi_{j}(\theta)d\nu_{m_{k}}(\theta) - \int \chi_{j}(\theta)d\nu_{m'_{k}}(\theta) \right | < \frac{1}{k}.
\end{equation*}
It then readily follows that $\nu_{m'_{k}} \Rightarrow \gamma_{2}$, and hence
we establish \eqref{eq:nunk*m'k=nunk*num'k}, which in turn implies
that $\gamma_{1}\bigstar \gamma_{2} \in \Ac$.

\end{proof}

\subsection{Cantor sets are attainable }
\label{sec:cantor}
By Proposition \ref{prop:A monoid}, $\Ac$ is closed under convolution,
it contains ~\cite{KKW} uniform measures
supported on symmetric intervals $[-\theta,\theta]$, as well as symmetric sums
$(\delta_\theta + \delta_{-\theta})/2$ for all $\theta>0$.
Thus, by using an ``additive'' construction of Cantor sets, we easily see
that uniform measures supported on Cantor sets are attainable.

Namely, given $\theta>0$, let $C_{n,\theta}$ be the $n$-th
level Cantor set obtained by starting with the interval $[-\theta,
\theta]$ and deleting the middle third part of the interval:
$C_{0,\theta}$ consists of one closed interval $[-\theta, \theta]$,
and $C_{n+1,\theta} \subset C_{n,\theta} $ is the union of
the $2^{n+1}$ intervals obtained by removing the middle third in each
of the $2^{n}$ intervals that $C_{n,\theta}$ consists of.  Now,
\begin{equation}
\label{eq:Cn+1 rec}
C_{n+1,\theta}
= (C_{n,\theta/3} - 2\theta/3)  \sqcup (C_{n,\theta/3} + 2\theta/3),
\end{equation}
where $\sqcup$ denotes disjoint union, and $C_{n+1,\theta/3} + \alpha$
denotes the translation of the set $C_{n+1,\theta/3}$ by $\alpha$.

Since $C_{0,\theta}$ is a symmetric interval, the measure corresponding to
its characteristic function is attainable, as mentioned above.
Further, since convolving $(\delta_\theta + \delta_{-\theta})/2$ with
a uniform measure having support on some set $D$ yields a measure with
support on $(D + \theta) \cup (D-\theta)$, uniform measures supported
on $C_{n,\theta}$ are attainable by induction, via \eqref{eq:Cn+1 rec}.
Letting $n \to \infty$ we find that measures with uniform support on
Cantor sets are attainable.

\subsection{Proof of Proposition \ref{prop:attain<=>strongly}}

\label{sec:attain<=>strongly}

\begin{proof}

We are going to make use of a (de-symmetrized)
Cilleruelo sequence $n_{j}$, i.e. $\nu_{n_{j}}\Rightarrow\delta_{0}$ and $r_{2}(n_{j})\rightarrow\infty$.
Let $\mu\in\Ac$ be an attainable measure and assume
that $\nu_{m_{j}}\Rightarrow\mu$. Using the same idea as in the course of proof
of Proposition \ref{prop:A monoid} above we may assume with no loss of
generality
that $(n_{j},m_{j})=1$ are co-prime (recall that $\{ n_{j}\}$ is a Cilleruelo sequence of our choice).
Then $$\nu_{m_{j}\cdot n_{j}} = \nu_{m_{j}}\bigstar \nu_{n_{j}} \Rightarrow \mu \bigstar \delta_{0} = \mu,$$
and $$r_{2}(m_{j}\cdot n_{j})/4 = r_{2}(m_{j})\cdot r_{2}(n_{j})
\rightarrow\infty,$$
so that the sequence $\{n_{j}\cdot m_{j}\}$ is as required.

\end{proof}

\section{Proof of Theorem \ref{thm:muhat-region}: measure classification for $x>\frac{1}{3}$}
\label{sec:meas-class x>1/3}

\label{sec:muhat-region}

\subsection{Some conventions related to Fourier Analysis}

We adapt the following conventions. The $k$-th Fourier
coefficient of
a measure $\mu\in\Pc$ is given by
\begin{equation*}
\widehat{\mu}(k)= \int\limits_{\TT^{1}}\cos(k\theta)d\mu(\theta);
\end{equation*}
clearly $|\widehat{\mu}(k)|\le 1$.
The convolution of two probability measures $\mu,\mu'\in\Pc$ is the probability measure
$\mu\bigstar\mu'$ defined as
\begin{equation*}
d(\mu\bigstar\mu') (\theta) = \int\limits_{\Sc^{1}}d\mu(\theta')d\mu'(\theta-\theta').
\end{equation*}

With the above conventions we have
\begin{equation*}
\widehat{\mu\bigstar\mu'}(k)=\widehat{\mu}(k)\cdot \widehat{\mu'}(k).
\end{equation*}
It is easy to compute the Fourier coefficients of $\upsilon_{\theta;M}$ as
in \eqref{eq:upsilonM,theta def} to be
\begin{equation*}
\widehat{\upsilon}_{\theta;M}(k) = \frac{1}{M+1}\sum\limits_{j=0}^{M}\cos((M-2j)k\theta)= G_{M+1}(k\theta),
\end{equation*}
where
\begin{equation}
\label{eq:GA(theta) def}
G_{A}(\theta) := \frac{\sin(A\theta)}{A\sin\theta};
\end{equation}
for $M=1$, $G_{2}(\theta)=\cos(\theta)$ is consistent with
\eqref{eq:upsilon_theta Fourier}.

By the definition of $\Ac$ and $\Ac_{k}=\Fc_{k}(\Ac)$ and in light of
Lemma \ref{lem:|thetap-theta|<eps}, we can describe $\Ac_{k}$ geometrically
as the smallest multiplicative set, closed in $\Pc_{k}$, containing all the
curves $$\left\{\gamma_{k;A}(\theta):=(G_{A}(\theta),\,\ldots,G_{A}(k\theta)):\:\theta\in
  [0,\pi] \right\}_{A\ge 2},$$ i.e. $\Ac_{k}$ is the closed
multiplicative subset of $\Pc_{k}$ generated by the above
curves. Similarly, the set corresponding to the square-free attainable
measures $\Ac^{0}_{k}$ is the smallest closed multiplicative set containing
the single curve $$\gamma_{k;2}(\theta)=(\cos(\theta),\,\ldots,\cos(k\theta)),$$ $\theta\in [0,\pi]$.

From this point on we will fix $k=2$ and suppress the $k$-dependence
in the various notation,
e.g. $\gamma_{A}$ will stand for $\gamma_{2;A}$.
The curves
\begin{equation}
\label{eq:gamma2A(theta) def}
\gamma_{A}(\theta):=(G_{A}(\theta),G_{A}(2\theta))
\end{equation}
for $2\le A\le 20$ are displayed in Figure
\ref{figg:prime-attainable}, separately for odd and even $M=A-1$.

\subsection{Proof of Theorem \ref{thm:muhat-region}}

The two statements of Theorem \ref{thm:muhat-region} are claimed in
Propositions \ref{prop:y<=M(x) x>1/3} and \ref{prop:x>1/3 under max attainable},
and proved in sections \ref{sec:y<=M(x) x>1/3} and \ref{sec:x>1/3 under max attainable} respectively.
Note that Proposition \ref{prop:x>1/3 under max attainable} yields attainable measures with the relevant Fourier
coefficients regardless whether $x>\frac{1}{3}$ or $x\le \frac{1}{3}$.

\begin{prop}
\label{prop:y<=M(x) x>1/3}
Points $(x,y)$ with $x>\frac{1}{3}$ corresponding to attainable
measures lie under the max curve,
i.e. if $(x,y)\in \Ac_{2}$ then
\begin{equation}
\label{eq:y<=M(x)}
y \le \Mcc(x),
\end{equation}
where $\Mcc(x)$ is given by \eqref{eq:max curve def}.
\end{prop}

\begin{prop}
\label{prop:x>1/3 under max attainable}
Given $x,y$ such that $|x| \leq 1$ and $$2x^{2}-1 \leq y
\leq  \Mcc(x) ,$$ there exists an
attainable measure $\mu$ such that $( \hat{\mu}(4), \hat{\mu}(8)) = (x,y)$.
\end{prop}

\subsection{Proof of Proposition \ref{prop:y<=M(x) x>1/3}:
attainable measures lie under the max curve for $x>1/3$.}

\label{sec:y<=M(x) x>1/3}

In what follows, by componentwise product we will mean
\begin{equation}
\label{eq:compwise prod}
(x_{1},y_{1})\cdot (x_{2},y_{2}) = (x_{1}\cdot x_{2}, y_{1} \cdot y_{2}).
\end{equation}

\begin{definition}[Totally positive and mixed sign points.]
\label{def:tot pos not}

Let $\Ac_{2}^{+}\subseteq \Ac_{2}$ be the set of {\bf totally
positive} attainable points admitting a representation as finite
componentwise products
\begin{equation}
\label{eq:(x,y)=prod(xi,yi)}
(x,y) = \prod\limits_{i=1}^{K}(x_{i},y_{i})
\end{equation}
of points $(x_{i},y_{i})=\gamma_{2;A_{i}}(\theta_{i})$ for some
$A_{i}\ge 2$, $\theta_{i}\in [0,\pi]$, so that for all $i \le K$ we
have $y_{i}>0$. Similarly, $\Ac_{2}^{-}\subseteq\Ac_{2}$ is the set of
{\bf mixed sign} attainable points admitting representation \eqref{eq:(x,y)=prod(xi,yi)}
with at least one $y_{i}<0$.
\end{definition}

Note that a point in $\Ac_{2}$ may be both totally positive and of
mixed sign, i.e. $\Ac_{2}^{+}$ may intersect
$\Ac_{2}^{-}$. Furthermore, a priori it may be in neither of these. However,
by the definition of $\Ac$, it is the closure of the union of the sets
defined:
\begin{equation}
\label{eq:Ac=bar(A+,A-)}
\overline{\Ac_{2}^{+}\cup \Ac_{2}^{-}} = \Ac.
\end{equation}
Therefore to prove the inequality \eqref{eq:y<=M(x)} on $\Ac_{2}$ it is sufficient
to prove the same for points in $\Ac_{2}^{+}$ and $\Ac_{2}^{-}$ separately.
These are established in Lemma \ref{lem:not tot pos under max curve}
and Proposition \ref{prop:prime curves x>1/3 below x^4}, proved in sections
\ref{sec:not tot pos under max curve} and \ref{prop:prime curves x>1/3 below x^4} respectively.

\begin{lem}
\label{lem:not tot pos under max curve}
If $(x,y)\in \Ac_{2}^{-}$ is a mixed sign attainable point
then $$y \le (2|x|-1)^{2}.$$
\end{lem}

\begin{prop}
\label{prop:prime curves x>1/3 below x^4}
Let $(x,y)=\gamma_{A}(\theta)$ for some $A\ge 2$ and $\theta \in [0,\pi]$
such that $x>\frac{1}{3}$. Then $y \leq x^{4}$.
\end{prop}

We are now in a position to prove Proposition \ref{prop:y<=M(x) x>1/3}.

\begin{proof}[Proof of Proposition \ref{prop:y<=M(x) x>1/3} assuming
Lemma \ref{lem:not tot pos under max curve} and Proposition
\ref{prop:prime curves x>1/3 below x^4}]
If the point $(x,y) \in \Ac_{2}^{-}$ is of mixed sign,
Lemma \ref{lem:not tot pos under max curve} applies and hence $y \leq (2|x|-1)^{2}$.
Otherwise, if the point is totally positive,
$$
(x,y) =  \left(\prod_{i} x_{i}, \prod_{i}y_{i}\right)
$$
where $(x_{i},y_{i})$ are prime power attainable, and $y_{i} \geq 0$
for all $i$.

Now, $|x_{i}| \leq 1$ for all $i$ since $x_{i}$ is a Fourier
coefficient of a probability measure, so if $|x|>1/3$ we must have
$|x_{i}| > 1/3$ for all $i$. By
Proposition \ref{prop:prime curves x>1/3 below x^4}, $y_{i} \leq
x_{i}^{4}$ for all $i$, and thus $y \leq x^{4}$.
Thus it follows that the statement \eqref{eq:y<=M(x)} of Proposition \ref{prop:y<=M(x) x>1/3} holds
on $\Ac_{2}^{+}\cup \Ac_{2}^{-}$ and thus on its closure, $\Ac_{2}$ (cf. \eqref{eq:Ac=bar(A+,A-)}).

\end{proof}

\subsection{Proof of Lemma \ref{lem:not tot pos under max curve}:
the mixed sign points $\Ac_{2}^{-}$ lie under the max curve}

\label{sec:not tot pos under max curve}

To pursue the proof of Lemma \ref{lem:not tot pos under max curve} we
will need some further notation.

\begin{notation}
\label{not:B1 B2 def}

Let $B_{1}\subseteq [-1,1] \times  [-1,1]$ be the set
\begin{equation*}
B_{1} = \{(x,y):\: x \in [-1/2,1/2],\, 0 \leq y \leq (2|x|-1)^{2} \},
\end{equation*}
and $B\subseteq [-1,1] \times  [-1,1]$ be the domain
\begin{equation*}
B_{2} = \{(x,y):\: x\in [-1/\sqrt{2},1/\sqrt{2}],\,  2x^{2}-1 \leq y \leq 0\}
\end{equation*}

\end{notation}

Recall the Definition \ref{def:tot pos not} of totally positive attainable points $\Ac_{2}^{+}$,
and componentwise product of points \eqref{eq:compwise prod}. It is
obvious that the points of either $B_{1}$ and $B_{2}$ are all lying
under the max curve, i.e. if $$(x,y)\in B_{1}\cup B_{2},$$ then
$$y \le \Mcc(x).$$ Therefore the following lemma implies
Lemma \ref{lem:not tot pos under max curve}.

\begin{lem}
\label{lem:mix sign in B1 or B2}
If $(x,y)\in \Ac_{2}^{-}$ is a mixed sign attainable point
then $$(x,y) \in B_{1}\cup B_{2}.$$
\end{lem}

To prove Lemma \ref{lem:mix sign in B1 or B2} we establish the
following two auxiliary lemmas whose proof is postponed until immediately
after the proof of Lemma \ref{lem:mix sign in B1 or B2}.

\begin{lem}
\label{lem:(x,y),y<0=>(x,y)in B}
If $(x,y) = (\hat{\mu}(1), \hat{\mu}(2))$ for $\mu$ some probability
measure on $\sone$ and $y \leq 0$, then $(x,y) \in B_{2}$.
\end{lem}

\begin{lem}
\label{eq:p1,p2 in B=> p1p2 in A}
If $p_{1}, p_{2} \in B_{2}$, then $p_{1} \cdot p_{2} \in B_{1}$.
\end{lem}

\begin{proof}[Proof of Lemma \ref{lem:mix sign in B1 or B2}
assuming the auxiliary lemmas.]

Let $$(x,y)\in \Ac_{2}^{-}$$ be given. First, if $(x,y)\in
\Ac_{2}^{-}$ with $y\le 0$, then $(x,y)\in B_{2}$ by Lemma
\ref{lem:(x,y),y<0=>(x,y)in B}; hence we may assume $y>0$. Let
$(x_{i},y_{i})$ be as in \eqref{eq:(x,y)=prod(xi,yi)}, which according
to the Definition \ref{def:tot pos not} have mixed signs.  Since $y
\geq 0$ we can in fact find $i \neq j$ for which $y_{i},y_{j}< 0$, and
without loss of generality we may assume that $(i,j)=(1,2)$.
Letting $$(\tilde{x},\tilde{y}) = \left(\prod_{k \neq 1,2} x_{k},
  \prod_{k \neq 1,2} y_{k}\right)$$ we find that $$(x,y) =
(x_{1},y_{1})\cdot (x_{2},y_{2}) \cdot (\tilde{x},\tilde{y}),$$ where
$\tilde{y} \in [0,1]$ and $\tilde{x} \in [-1,1]$.

We further note that both $(x_{1},y_{1})$ and $(x_{2},y_{2})$ lie in
$B_{2}$. Thus by Lemma \ref{eq:p1,p2 in B=> p1p2 in A},
$$(x_{1},y_{1})\cdot (x_{2},y_{2}) \in B_{1}.$$  Since $|\tilde{x}|,\tilde{y} \leq 1,$ the
result follows on noting that $B_{1}$ is mapped into itself by any map
of the form $$(x,y) \to (\alpha x, \beta y),$$ provided that
$$0 \leq |\alpha|, \beta \leq 1.$$
\end{proof}

\subsubsection{Proofs of the auxiliary lemmas \ref{lem:(x,y),y<0=>(x,y)in B} and \ref{eq:p1,p2 in B=> p1p2 in A}}

\begin{proof}[Proof of Lemma \ref{lem:(x,y),y<0=>(x,y)in B}]

The assumptions 
are equivalent to $(x,y)\in \Pc_{2}$ with $y\le 0$.
The statement
follows immediately upon using
the explicit description \eqref{eq:P2 Riesz} of $\Pc_{2}$:
\begin{equation*}
\Pc_{2}\cap \{ y\le 0\} = B_{2}.
\end{equation*}

\end{proof}

\begin{proof}[Proof of Lemma \ref{eq:p1,p2 in B=> p1p2 in A}]
  The case of either point having zero $y$-coordinate is trivial, so
  we may assume that both $p_{1},p_{2}$ have negative
  $y$-coordinates, and it suffices to prove the statement for points
  $p_{1},p_{2}$ having minimal $y$-coordinates, i.e.,
$$
p_{1} = (a, 2a^{2}-1), \quad p_{2} = (b, 2b^{2}-1),
$$
and we may further assume $ab\ne 0$ as otherwise the statement is trivial.

By symmetry it suffices to consider the case $a,b \in (0,1/\sqrt{2})$.
Thus, if we fix $c \in (0,1/2)$ it suffices to determine the maximum
of $$(2a^{2}-1)(2b^{2}-1)$$ subject to the constraint $ab = c$.  Taking
logs we find that the constraint is given by
$$
\log a + \log b = \log c
$$
and we wish to maximize
$$
\log( 1-2a^{2}) + \log(1-2b^{2}).
$$
Using Lagrange multipliers we find that all internal maxima satisfies
$$
(1/a,1/b) = \lambda \left( \frac{4a}{1-2a^{2}}, \frac{4b}{1-2b^{2}}\right)
$$
for some $\lambda \in \R$.  If $c = ab \neq 0$ we find that
$$
(1,1) = \lambda \left( \frac{4a^{2}}{1-2a^{2}}, \frac{4b^{2}}{1-2b^{2}}\right)
$$
and thus $\frac{4a^{2}}{1-2a^{2}} = \frac{4b^{2}}{1-2b^{2}}$ which
implies that $a^{2}=b^{2}$, and hence,
recalling that we assumed $a,b\ge 0$, it yields $a = b$.
In particular, any internal maximum gives a point $(a^{2},
(2a^{2}-1)^{2}) = (c,(2|c|-1)^{2}$), which lies on the boundary of $B_{1}$.
As mentioned earlier, for points on the boundary, the inequality
holds trivially.

\end{proof}

\subsection{Proof of Proposition \ref{prop:prime curves x>1/3 below x^4}:
totally positive points $\Ac_{2}^{+}$ corresponding to prime powers}

\label{sec:prime curves x>1/3 below x^4}

\begin{lem}
\label{lem:sinc-decreasing}
The function $\frac{\sin t  }{t}$ is decreasing and is $ \geq 0$ on
$[0,\pi]$.
\end{lem}
\begin{proof}
  Taking derivatives, this amounts to the fact that $\tan t > t$ on
  $(0,\pi/2)$.
\end{proof}

\begin{lem}
\label{lem:x-bigger-one-third}
If $A \geq 4$ and $|G_{A}(t)| \geq 1/3$ for $t \in [0,\pi/2]$,
then $t \leq \frac{\pi}{A}$.
For $A=3$, we have the further possibility that $t = 3 \pi/(2A)
  = \pi/2$.
\end{lem}
\begin{proof}
The inequality $\sin t \geq 2t/\pi$, valid for $t \in [0,\pi/2]$,
and strict except at the end points, gives that
$$
|G_{A}(t) |  = \left| \frac{\sin(A\theta)}{A\sin\theta} \right|
\leq \frac{1}{A \sin t} \leq \frac{1}{A \cdot
  \frac{2}{\pi} t}
$$
and hence $|G_{A}(t)| < 1/3$ for $t > 3 \pi/(2A)$, for any
$A>0$.
It thus suffices to consider $t \in [0, 3 \pi/(2A)]$.

Consider first the case $A=3$. We begin by showing that $G_{3}(t)$ is
decreasing on $[0,\pi/2]$. Taking derivatives, this amounts to
the fact that
that $$3 \tan t \neq \tan 3t$$ on $[0,\pi/2]$.  Now, since
$G_{3}(\pi/3) = 0$ and $G_{3}(\pi/2) = -1/3$ and $G_{3}$ is
decreasing, we find that the only possibility for $|G_{3}(t)| = 1/3$
and $t \in [\pi/3,\pi/2]$ is $t = \pi/2$.  Thus, any other solution must lie
in $[0,\pi/3] = [0,\pi/A]$.

For $A \geq 4$, note that
\begin{equation}
\label{eq:crap1}
\left|\frac{\sin At}{A \sin t} \right|
= \left|\frac{\sin(At)/(At)}{\sin(t)/t} \right|
<
\left|\frac{\sin(At)/(At)}{\sin(At/3)/(At/3)} \right|
\end{equation}
(for $t \leq 3\pi/(2A)$ we have $At/3 \leq \pi/2$, hence
$$
|\sin(At/3)/(At/3) | \leq |\sin(t)/t|,
$$
since $(\sin x)/x$ is decreasing on the interval $[0,\pi]$ by
Lemma \ref{lem:sinc-decreasing}.)

Taking $s = At/3$, the RHS of (\ref{eq:crap1}) becomes
$$
\frac{(\sin 3s)/3s}{(\sin s) /s} = \frac{\sin 3s}{3 \sin s}
$$
and $t \leq 3\pi/(2A)$ implies that $s \leq \pi/2$.  For this range of
$s$, by the first part of the lemma, we find that $\left|\frac{\sin 3s}{3
  \sin s}\right| \geq 1/3$ implies that either $s=\pi/2$ or $s \leq \pi/3$,
which in turn implies that $t =3\pi/(2A)$ or $t \leq \pi/A$.  Noting
that the first possibility is ruled out by the strict inequality in
(\ref{eq:crap1}), the proof is concluded.

\end{proof}

We proceed to characterize points lying on curves $\{
(x,y)=\gamma_{A}(t)\}_{A\ge 2}$, for which $x > 1/3$ and $y \geq 0$,
showing that any such point satisfies $y \leq x^{4}$. We begin with
the following key Lemma.
\begin{lem}
\label{lem:key}
For $t \in (0,\pi/2]$, define
\begin{equation}
  \label{eq:h(t)-definition}
h(t) := \frac{t^{3}\cos t}{\sin^{3}t}
\end{equation}
and extend $h$ to $[0,\pi/2]$ by continuity.
Then $h(t)$ is decreasing on $[0, \pi/2]$.
\end{lem}
\begin{proof}

We have
\begin{equation*}
\begin{split}
h'(t) &= \frac{t^{2}\sin^{2}(t) \left(  \sin(t) \cos(t)-t\sin^{2}(t)-3t\cos^{2}(t)\right)}{\sin^{6}t},
\end{split}
\end{equation*}
and it is enough to show that
\begin{equation}
\label{eq:h'<0 denominator}
\begin{split}
\sin(t) \cos(t)-t\sin^{2}(t)-3t\cos^{2}(t) < 0
\end{split}
\end{equation}
for $t\in (0,\pi/2)$. Since for $t=0$ the expression on the left hand
side of \eqref{eq:h'<0 denominator} vanishes it is sufficient to show
that its derivative is strictly negative on
$\left(0,\frac{\pi}{2}\right)$.  We find that
\begin{multline*}
\left(
\sin(t) \cos(t)-t\sin^{2}(t)-3t\cos^{2}(t) \right)' =
\\=
4\sin(t)(t\cos(t)-\sin(t))
=4\sin(t) \cos(t)(t-\tan{t}) < 0
\end{multline*}
since $\tan(t)>t$ on $\left(0,\frac{\pi}{2}\right)$.

\end{proof}

\begin{proof}[Proof of Proposition \ref{prop:prime curves x>1/3 below x^4}.]

If $A=2$, the points lying on the curve $\gamma_{2}$
are of the form $$(x,y)=\gamma_{2}(t)= (t, 2t^{2}-1),$$ and it is
straightforward to check that $2t^{2}-1 \leq t^{4}$.
For $A \geq 3$, since we assume that $x > 1/3$ and
$$
(x,y) = (G_{A}(t), G_{A}(2t) ),
$$
Lemma \ref{lem:x-bigger-one-third} implies that $t \leq \pi/A$.  In
fact, $t \leq \pi/(2A)$, as we assume that $y \geq 0$.
Hence it is sufficient to show  that
$$
\frac{\sin 2At}{A \sin 2t} \leq \left( \frac{\sin At}{A \sin
    t}\right)^{4}
$$
holds for $t \in [0, \pi/(2A)]$.

This in turn  is equivalent (note that all individual trigonometric terms are
non-negative since $t \in [0, \pi/(2A)]$) to
$$
A^{3} \cos At \sin^{3} t \leq \sin^{3} At \cos t
$$
which is equivalent to
$$
\frac{(At)^{3} \cos At}{\sin^{3} At} \leq
\frac{t^{3  }\cos t}{\sin^{3} t}.
$$
Setting $$s = At\in [0,\pi/2],$$ we find that this is equivalent to
$$
\frac{s^{3} \cos s}{\sin^{3} s} \leq
\frac{(s/A)^{3  }\cos s/A}{\sin^{3} s/A},
$$
or, equivalently on recalling (\ref{eq:h(t)-definition}), that
$$
h(s) \leq h(s/A).
$$
which, as $A>1$,  follows from Lemma~\ref{lem:key}.

\end{proof}

\subsection{Proof of Proposition
\ref{prop:x>1/3 under max attainable}: all points under the max curve are attainable}

\label{sec:x>1/3 under max attainable}

\begin{lem}
\label{lem:x^4 attainable}
The curve $\{ (x,x^4): x\in [0,1]\}$ is square-free attainable,
i.e. all the points on this curve correspond to at least one attainable measure.
\end{lem}

\begin{proof}[Proof of Proposition \ref{prop:x>1/3 under max attainable} assuming Lemma
\ref{lem:x^4 attainable}]

By the definition of the max curve \eqref{eq:max curve def} it is sufficient to prove that if $(x_{0},y_{0})$
is lying under one of the curves $y=x^{4}$ and $y=(2|x|-1)^{2}$ then $(x_{0},y_{0})\in \Ac_{2}$
is attainable; with no loss of generality we may assume that $x_{0}\ge 0$.
Now we know that the parabola $\{ (t,2t^{2}-1)\}_{t\in [0,1]}$ is attainable, and
from Lemma \ref{lem:x^4 attainable} so is the curve $\{(x,x^{4})\}_{x\in [0,1]}$.

It then follows by multiplicativity of $\Ac_{2}$
that all the points of the form $$(x_{0},y_{0}) = (x,x^{4}) \cdot (t,2t^{2}-1) $$
are attainable (recalling the notation \eqref{eq:compwise prod} for componentwise multiplication).
On the other hand it is clear
that the union of
the family of the parabolas
$$
\{(xt,x^{4}(2t^{2}-1)) : t \in [0,1]    \},
$$
as  $x$ ranges over $[0,1]$,
is exactly the set
$$ \{(x,y):\: x\in [0,1], \, 2x^{2}-1 \le y\le x^{4}  \}.$$

Concerning points under the other curve $y=(2x-1)^2$
we may employ the multiplicativity of $\Ac_{2}$ again to yield that the curve
$$\{(x^{2},(2x^{2}-1)^{2})\}_{x\in [0,1]}$$
is attainable;  this curve in turn can be re-parameterized as $\{
(t,(2t-1)^{2})\}_{t\in[0,1]}$. A similar argument to the
above shows that function $$(x,t)\mapsto (x,(2x-1)^{2})\cdot
(t,2t^{2}-1)$$ 
maps
$[0,1]^{2}$ onto the domain $$\{(x,y):\: x\in [0,1],\,  2x^{2}-1 \le y\le (2x-1)^{2} \},$$
i.e. as the parameter $x$ varies along $[0,1]$ the parabolas $$\{(xt,(2x-1)^{2}\cdot (2t^{2}-1))\}$$
tessellate the domain under the curve $y=(2x-1)^{2}$, $x\in [0,1]$.
Hence all the points under the latter curve are attainable, as claimed.

\end{proof}

\begin{proof}[Proof of Lemma \ref{lem:x^4 attainable}]
We start with the case $x \geq 0$.
We know that the curve $\{(x, 2x^2-1)\}_{x\in [-1,1]}$ is attainable as a re-parametrization
of  $(\cos\theta, \cos 2 \theta)$ (i.e. all the points on that curve correspond
to attainable measures), hence for $n\ge 1$ the curve $\{(x^n, (2x^2-1)^n)\}$ is
attainable by the multiplicativity (cf. Proposition \ref{prop:A monoid}).
Fix $\alpha >0$, and take $x = x_n = e^{-\alpha/ n}$.
Thus
$$
( e^{-\alpha}, (2 e^{-2 \alpha/n} - 1)^n )
$$
is attainable for every $\alpha> 0$ and $n \ge 1$.

Upon using Taylor series, we find that, as $n \to \infty$,
$$
(2 e^{-2 \alpha/n} - 1)^n =
\left(2 \left(1-\frac{2\alpha}{n} + O\left(\frac{1}{n^2}
    \right)\right) - 1\right)^n =
$$
$$
\left(1 - \frac{4 \alpha}{n}
  +O\left(\frac{1}{n^2} \right)\right)^n
= e^{-4 \alpha} + o(1).
$$
Since this holds for any fixed $\alpha >0$, bearing in mind that $\Ac$ is closed
in $\Pc$ (and hence the set $\Ac_{2}\subseteq [-1,1]^{2}$ is closed in the usual sense), we indeed find that the
curve $(x,x^4)$ lies in the attainable set for every $x \in (0,1)$. It is
easy to see that also $(0,0)$ and $(1,1)$ are attainable.
By reflecting the curve $(x,x^{4})$ (for $x \geq 0$) in the $x$-axis
(using that $(-1,1)$ is attainable and multiplying) we find that
$(x,x^{4})$ is attainable for $x \in [-1,1]$.
\end{proof}

\section{Proof of Theorem \ref{thm:Ac2 top bar}: fractal structure for  $x<\frac{1}{3}$}

\label{sec:fractal x>1/3}

It is obvious that the second assertion of Theorem \ref{thm:Ac2 top bar}
implies the first part, so we only need to prove the second one.
However, since the proof of the second assertion is fairly complicated
we give a brief outline of how the first assertion can be deduced,
and then indicate how to augment the argument to give the second assertion.


We are to understand the closure of all the points $(x,y)$ of the form
\begin{equation}
\label{eq:(x,y)=GA prod}
(x,y)=\prod\limits_{i=1}^{K}(G_{A_{i}}(t_{i}),G_{A_{i}}(2t_{i}))
\end{equation}
with $A_{i}\ge 2$ arbitrary integers.
Using that $G_{A}(\pi/2+t)$ is either even or odd (depending on the
parity of $A$) and that $G_{A}(2(\pi/2+t))$ is even, together with
signs of $x$-coordinates
being irrelevant (since $(x,y)$ is attainable if and only if $(-x,y)$
is attainable) we may assume that $t_{i}\in
\left[0,\frac{\pi}{2}\right]$ for all $i$.  A curve
$(x_{0},y_{0})=(G_{A_{0}}(t_{0}),G_{A_{0}}(2t_{0}))$ turns out to
intersect the line $y=1$ with $|x|\le \frac{1}{3}$ only for $A_{0}$
odd, and further forces $t_{0}=\frac{\pi}{2}$, and $x=\pm
\frac{1}{A}$.  Hence the point $(x,y)$ as in \eqref{eq:(x,y)=GA prod}
satisfies $y=1$ only
for $A_{i}$ odd and $t_{i}=\frac{\pi}{2}$ for all $i\le K$, whence
$(x,y)=(\pm \frac{1}{A},1)$ with $A=\prod\limits_{i=1}^{K}A_{i}$.

To prove the second assertion we investigate a (fairly large)
neighborhood of the point $(\frac{1}{A},1)$; given an odd $A$ we
consider all finite products \eqref{eq:(x,y)=GA prod} with
$A=\prod_{i=1}^{K}A_{i}$ and $t_{i}\approx \frac{\pi}{2}$ (and $A_{i}
\geq 3$.) We will prove that all products $(x,y)$ of this form will
stay between two curves defined below; after taking logarithms this
will amount to the fortunate log-convexity of the curves
$(G_{A_{0}}(t),G_{A_{0}}(2t))$, $A_{0}\ge 3$ odd, in the suitable
range (see Lemma \ref{lem:GA curv convex} below). We argue that this
property is invariant with respect to multiplying by curves
$(G_{A_{1}}(t),G_{A_{1}}(2t))$ for $A_{1}\ge 2$ even, and also for odd
$A_{1} \ge 3$ for $t$ near $\pi/2$.

\subsection{Proof of the second assertion of Theorem \ref{thm:Ac2 top bar}}
To prove the main result of the present section we will need the
following results.

\begin{prop}
\label{prop:x<1/3 below mix cur}
Let $\{ A_{i}\}_{i}$ be a finite collection of integers $A_{i}\ge 2$,
and consider a point $(x,y)$ of the form
\begin{equation}
\label{eq:(x,y)=GA fin prod}
(x,y)  =
\left( \prod_{i} G_{A_{i}}(t_{i}), \prod_{i} G_{A_{i}}(2t_{i}) \right),
\end{equation}
where all $t_{i} \in [0, \pi/2]$. Assume that one of the following is satisfied:

\begin{itemize}
\item There exists $i$ such that $A_{i} \geq 3$ is odd and $t_{i} \in
  [\pi/(2A_{i}), \pi/2 - \pi/(2A_{i})]$.

\item There exists $i$ such that $A_{i}$ is even and $t_{i} \geq
  \pi/(2A_i)$.

\end{itemize}
Then necessarily
\begin{equation*}
y\le (2|x|^2-1).
\end{equation*}

\end{prop}

The proof of Proposition \ref{prop:x<1/3 below mix cur} is
postponed to Appendix
\ref{apx:below mixed big small}.

\begin{prop}
\label{prop:f2 finite}

Let $A\ge 3$ be an odd number, and $$A=\prod\limits_{i=1}^{K}A_{i}$$ an arbitrary (fixed) factorization of $A$
into (not necessarily co-prime) integers $A_{i}\ge 3$. For $x\le \frac{1}{A}$ define
\begin{equation}
\label{eq:gAi(x) def}
g_{\{A_{i}\}}(x) = \sup\limits_{(t_{i})_{i} \in \mathcal{X}_{\{A_{i}\}}(x)} \prod\limits_{i=1}^{K} G_{A_{i}}(2t_{i}),
\end{equation}
the supremum taken w.r.t. all $(t_{i})_{i\le K}$ lying in
\begin{equation}
\label{eq:Xc def}
\mathcal{X}_{\{A_{i}\}}(x):= \left\{(t_{i})_{i}:\forall i\le K, \, t_{i}\in \left[\frac{\pi}{2}-\frac{\pi}{2A_{i}},\pi/2\right],
\, \left|\prod\limits_{i=1}^{K}G_{A_{i}}(t_{i})\right| = x \right\}.
\end{equation}
Then for every $0<x < \frac{1}{A}$ there exists an index
$i_{0}=i_{0}(x)\le K$ and
$t\in [\frac{\pi}{2}-\frac{\pi}{2A_{i}},\pi/2]$ such
that\footnote{The reason for
  $\frac{A_{i_{0}}}{A}|G_{A_{i_{0}}}(t)|$ appearing is that the
  supremum is attained by having $t_{i}=0$ for $i \neq i_{0}$ and
  hence $\prod_{i\neq i_{0}} G_{A_{i}}(0) = \prod_{i\neq i_{0}}
  1/A_{i} = A_{i_{0}}/A $.}
$$(x,g_{\{A_{i}\}}(x)) = \left(\frac{A_{i_{0}}}{A}|G_{A_{i_{0}}}(t)|,G_{A_{i_{0}}}(2t)\right),$$
and moreover the map $x\mapsto i_{0}(x)$ is piecewise
constant.
In particular, the function $g_{\{A_{i}\}}(x)$ is continuous,
analytic in some (left) neighbourhood of $x=\frac{1}{A}$, and
piecewise analytic on $(0,\frac{1}{A}]$.
\end{prop}

We may finally define the function $f_{2;k}$ introduced in Theorem
\ref{thm:Ac2 top bar}.

\begin{definition}
\label{def:f2 def}
Given $k\ge 1$ define
\begin{equation*}
f_{2;k}(x)=\max\limits_{\prod\limits_{i=1}^{K}A_{i}=2k+1} g_{\{A_{i}\}}(x),
\end{equation*}
the maximum taken w.r.t. all non-trivial factorizations of $2k+1$, i.e.,
all sets of (odd) integers $\{ A_{i}\}_{i=1}^{K}\subseteq\Z_{\ge 3}$, whose product is $2k+1$.
\end{definition}

\begin{rem}Recall the assumption that $0 < x < 1/3$.

\begin{enumerate}

\item By the definition of $g_{\{A_{i}\}}$ and $f_{2;k}$, if $(x,y)$ is of the form
$$(x,y)=\prod\limits_{i=1}^{K}(|G_{A_{i}}(t_{i})|,G_{A_{i}}(2t_{i}))$$
with all $A_{i}\ge 3$ odd, then necessarily
\begin{equation}
\label{eq:fin prod y<=f2(x)}
y\le g_{\{A_{i}\}_{i\le K}}(x) \le f_{2;k}(x),
\end{equation}
where $k$ is defined as in $$\prod\limits_{i=1}^{K}A_{i}=2k+1.$$

\item Proposition \ref{prop:f2 finite} implies that for $k\ge 1$ and $x<\frac{1}{2k+1}$,
\begin{equation*}
f_{2;k}(x)=\max\limits_{1<A | 2k+1 }\max\limits_{\left\{t: \left|\frac{A}{2k+1}G_{A}(t)\right|=x\right\} } G_{A}(2t),
\end{equation*}
a maximum w.r.t. all (odd) divisors $A>1$ of $2k+1$; the latter yields
an algorithm
for computing $f_{2;k}(x)$, reducing the original problem into
maximizing a finite set of numbers.

\end{enumerate}

\end{rem}

The following $3$ results will be proven in Appendix \ref{apx:GA curv convex}.

\begin{lem}
\label{lem:GA analyt par}
Let $A\ge 3$ be an odd integer, and $\eta_{A}$ be the parametric curve in $\R^{2}$ defined by
\begin{equation}
\label{eq:etaA def}
\eta_{A}(t) = (\eta_{A;1}(t),\eta_{A;2}(t))=(\log(A\cdot |G_{A}(t)|),\log(G_{A}(2t))),
\end{equation}
for $t\in \left(\frac{\pi}{2}-\frac{\pi}{2A},\frac{\pi}{2}\right]$.
Then we may re-parameterize $\eta$ as $(z,h_{A}(z))$ for some analytic function $h:(-\infty,0)\rightarrow \R_{\le 0}$ with $h(0)=0$,
and moreover $0< h'(z) \le \frac{4}{3}$ everywhere in the above range.
\end{lem}

\begin{cor}
\label{cor:y^4/3>=1/Ax}
Let $\{ A_{i}\}_{i=1}^{K}\subseteq \Z_{\ge 3}$ be a set of odd integers, $A=\prod\limits_{i=1}^{K}A_{i}$, and $(x,y)$ of the form
\begin{equation*}
(x,y) = \prod\limits_{i=1}^{K}(G_{A_{i}}(t_{i}),G_{A_{i}}(2t_{i})),
\end{equation*}
such that for all $i\le K$ we have $t_{i}\in \left[\frac{\pi}{2}-\frac{1}{2A_{i}},\frac{\pi}{2}\right]$.
Then necessarily $$ y\ge (Ax)^{4/3} .$$
\end{cor}

\begin{lem}
\label{lem:convexity log(2exp(2z)-1)}
For every $x_{1},x_{2} \in [0,1]$ the following inequality holds:
\begin{equation}
\label{eq:convexity log(2exp(2z)-1)}
(2x_{1}^{2}-1)\cdot (2x_{2}^{2}-1) \ge (2(x_{1}x_{2})^{2}-1).
\end{equation}
\end{lem}

We are finally in a position to prove Theorem \ref{thm:Ac2 top bar}
(with the first assertion following from the second.)

\begin{proof}[Proof of the second assertion of Theorem \ref{thm:Ac2 top bar}
assuming the results above]

We first prove that any point $(x,y)\in \Ac_{2}$ with $0<x<\frac{1}{3}$ either satisfies $y\le (2x-1)^{2}$ or
$(x,y)\in \Dc_{0,x_{k}}(f_{1;k},f_{2;k})$ for some $k\ge 1$, i.e. establish the inclusion $\subseteq$ of \eqref{eq:Ac2=mix, spikes}.
Since $\Ac_{2}$ is the closure (in $\R^{2}$) of the set of finite products
\begin{equation}
\label{eq:(x,y)=fin prod GA}
(x,y)=\prod\limits_{i=1}^{K}(G_{A_{i}}(t_{i}),G_{A_{i}}(2t_{i})),
\end{equation}
with some $A_{i}\ge 2$, $t_{i}\in [0,\pi]$, and the set on the r.h.s. of
\eqref{eq:Ac2=mix, spikes} is closed in $\{x>0\}$, it is sufficient to prove it for the finite products
\eqref{eq:(x,y)=fin prod GA}.

Thus let $(x,y)$  be given by a finite product \eqref{eq:(x,y)=fin prod GA};
by the invariance of $\Ac_{2}$ w.r.t. $x\mapsto -x$ we may assume that
all $t_{i}$, $i\le K$ satisfy
$t_{i}\in [0,\pi/2]$. If there exists either an odd $A_{i}$ such that
$t_{i}\in [\frac{\pi}{2A_{i}}, \frac{\pi}{2}-\frac{\pi}{2A_{i}}]$, or an even $A_{i}$ such that
$t_{i}\in [\frac{\pi}{2A_{i}}, \frac{\pi}{2}]$, then one of the sufficient conditions of Proposition
\ref{prop:x<1/3 below mix cur} is satisfied, implying that $y\le (2x-1)^{2}$, so that our present statement
holds.

We may then assume that for all odd $A_{i}$ we have either $t_{i}\in [0,\frac{\pi}{2A_{i}})$ or
$t_{i}\in \left(\frac{\pi}{2}-\frac{\pi}{2A_{i}},\frac{\pi}{2}\right]$, and for all even $A_{i}$ we have
$t_{i}\in  \left[0,\frac{\pi}{2A_{i}}\right)$. Up to reordering the
indexes, we may assume that $K=K_{1}+K_{2}$ with $K_{1} > 0$, and where
all the $A_{i}$ with $i\le K_{1}$ are odd and $t_{i}\in \left[\frac{\pi}{2}-\frac{\pi}{2A_{i}},\frac{\pi}{2}\right]$,
and for all $K_{1}+1\le i\le K_{2}$ we have $t_{i}\in [0,\frac{\pi}{2A_{i}}]$, whether the corresponding $A_{i}$ is odd or even.
Let
\begin{equation}
\label{eq:A prod odd Ai}
A=\prod\limits_{i=1}^{K_{1}}A_{i}=2k+1.
\end{equation}
be the product of the first $K_{1}$ odd $A_{i}$. We claim that,
with $k$ as defined in \eqref{eq:A prod odd Ai}, necessarily
\begin{equation}
\label{eq:f1(x)<=y<=f2(x)}
f_{1;k}(x) \le y\le f_{2;k}(x).
\end{equation}

Define $$(x_{0},y_{0})= \prod\limits_{i=1}^{K_{1}}(G_{A_{i}}(t_{i}),G_{A_{i}}(2t_{i}))$$
and $$(x_{1},y_{1})= \prod\limits_{i=K_{1}+1}^{K_{1}+K_{2}}(G_{A_{i}}(t_{i}),G_{A_{i}}(2t_{i})),$$
so that
\begin{equation}
\label{eq:(x,y)=(x0,y0)*(x1,y1)}
(x,y)=(x_{0},y_{0})\cdot (x_{1},y_{1}).
\end{equation}
By \eqref{eq:fin prod y<=f2(x)},
we have $y_{0}\le g_{\{A_{i}\}_{i\le K_{1}}}(x_{0})$, and by Proposition \ref{prop:f2 finite} there exists $i_{0} \le K_{1}$
and $t_{0}\in \left(\frac{\pi}{2}-\frac{\pi}{2A_{i_{0}}},\frac{\pi}{2}\right]$, so that
\begin{equation}
\label{eq:x0=Ai0/A GAi0(0)}
x_{0}=\frac{A_{i_{0}}}{A}|G_{A_{i_{0}}}(t_{0})|
\end{equation}
and $g_{\{A_{i}\}_{i\le K_{1}}}(x_{0}) = G_{A_{i_{0}}}(2t_{0})$;
we then have
\begin{equation}
\label{eq:y0<=GAi0(2t0)}
y_{0} \le G_{A_{i_{0}}}(2t_{0}).
\end{equation}

For the sake of brevity of notation we assume with no loss of
generality that $i_{0}=1$,
and consider the curve
$\eta_{A_{1}}$ in $\R_{> 0}^{2}$ as in Lemma \ref{lem:GA analyt par};
by the virtue of the latter lemma we may re-parameterize
$\eta_{A_{1}}$ as $(z,h_{A_{1}}(z))$ in the range $z\in (-\infty,0]$,
and $0< h_{A_{1}}'(x) \le \frac{4}{3}$ everywhere.
Hence, on noting that all the logarithms involved are {\em
  negative}, the mean value theorem gives that
\begin{equation}
\label{eq:hA1(x0x1)>=hA1(x0)+4/3x1}
h_{A_{1}}(\log (Ax_{0}x_{1})) =h_{A_{1}}(\log(Ax_{0})+\log(x_{1})) \ge  h_{A_{1}}(\log(Ax_{0}))+\frac{4}{3}\log(x_{1}).
\end{equation}
Note that by \eqref{eq:x0=Ai0/A GAi0(0)} and the definition of $h_{A_{1}}$ as a re-parametrization of \eqref{eq:etaA def}, we have
$$h_{A_{1}}(\log(Ax_{0})) = h_{A_{1}}(\log(A_{1}|G_{A_{1}}(t_{0})|) =  \log G_{A_{1}}(2t_{0})$$
(recall that we assumed that $i_{0}=1$).

Substituting the latter into \eqref{eq:hA1(x0x1)>=hA1(x0)+4/3x1} it
implies that
there exist a number $\theta_{1}\in
\left(\frac{\pi}{2}-\frac{\pi}{2A_{1}},\frac{\pi}{2}\right]$
satisfying $A_{1}G_{A_{1}}(\theta_{1})=Ax_{0}x_{1}$ (note that $x_{0}
\in [0,1/A]$) and
$$\log(G_{A_{1}}(2\theta_{1})) \ge \log G_{A_{1}}(2t_{0})+\frac{4}{3}\log(x_{1}).$$
Equivalently,
\begin{equation}
\label{eq:GA1(theta1)=A/A1x0x1}
G_{A_{1}}(\theta_{1})=\frac{A}{A_{1}} x_{0}x_{1}
\end{equation}
and
\begin{equation}
\label{eq:GA1>=y0x1^4/3}
G_{A_{1}}(2\theta_{1}) \ge G_{A_{1}}(2t_{0})\cdot x_{1}^{4/3} \ge y_{0}\cdot x_{1}^{4/3},
\end{equation}
by \eqref{eq:y0<=GAi0(2t0)}.

Note that for the choice $t_{1}=\theta_{1}$ and $t_{i}=\frac{\pi}{2}$ for $2 \le i \le  K_{1}$,
we have
\begin{equation}
\label{eq:prodGAi(ti)=x0x1}
\left| \prod\limits_{i=1}^{K_{1}}G_{A_{i}}(t_{i}) \right| = \frac{A}{A_{1}} x_{0}x_{1}\cdot \prod\limits_{i=2}^{K_{1}}\frac{1}{A_{i}} = x_{0}x_{1},
\end{equation}
by \eqref{eq:GA1(theta1)=A/A1x0x1} and \eqref{eq:A prod odd Ai}. Now, bearing in mind \eqref{eq:(x,y)=(x0,y0)*(x1,y1)},
as $g_{\{A_{i}\}_{i\le K_{1}}}(x)$ is defined to be the supremum of
all the expressions \eqref{eq:gAi(x) def} with $\{t_{i}\}_{i\le K_{1}}$ satisfying \eqref{eq:prodGAi(ti)=x0x1},
and recalling Definition \ref{def:f2 def} of $f_{2;k}(x)$,
\eqref{eq:GA1>=y0x1^4/3} implies that
\begin{equation}
\label{eq:f2(x)<=f2(x0)x1^4/3}
f_{2;k}(x) \ge g_{\{A_{i}\}_{i\le K_{1}}}(x)\ge y_{0}\cdot x_{1}^{4/3}.
\end{equation}
On the other hand, we use the upper bound
\begin{equation}
\label{eq:y1<=x1^4}
y_{1}\le x_{1}^{4}
\end{equation}
of Lemma \ref{prop:prime curves x>1/3 below x^4}
(valid for $(x_{1},y_{1})$). The inequality \eqref{eq:y1<=x1^4}
together with \eqref{eq:f2(x)<=f2(x0)x1^4/3} and the fact that
$x^{4/3} > x^{4}$ for $x < 1$ yield that
\begin{equation*}
f_{2;k}(x) \ge y_{0}\cdot x_{1}^{4/3} \ge y_{0}\cdot x_{1}^{4} \ge y_{0}\cdot y_{1} = y,
\end{equation*}
as in \eqref{eq:(x,y)=(x0,y0)*(x1,y1)},
which is the second inequality of \eqref{eq:f1(x)<=y<=f2(x)}.

To prove the first inequality of \eqref{eq:f1(x)<=y<=f2(x)} we use Corollary \ref{cor:y^4/3>=1/Ax}
to yield $y_{0} \ge (Ax_{0})^{4/3}$ with $A$ as in \eqref{eq:A prod odd Ai}.
These combined imply
\begin{equation*}
y = y_{0}\cdot y_{1} \ge (Ax_{0})^{4/3}\cdot (2x_{1}^{2}-1) \ge
(Ax_{0})^{4}\cdot (2x_{1}^{2}-1) \ge (2(Ax_{0})^{2}-1) \cdot (2x_{1}^{2}-1)
\end{equation*}
where we used the obvious inequality $x^{4}\ge 2x^{2}-1$, valid on $[-1,1]$.
Finally, an application of the inequality \eqref{eq:convexity
  log(2exp(2z)-1)} of Lemma
\ref{lem:convexity log(2exp(2z)-1)} yields
\begin{equation*}
y\ge 2(Ax_{0}x_{1})^{2}-1 = 2A^{2}\cdot x^{2}-1 = f_{1;k}(x),
\end{equation*}
by the definition \eqref{eq:f1 def} of $f_{1;k}$, and recalling that $x_{k}=\frac{1}{2k+1}$.

Conversely, we need to prove that any point $(x,y)$ satisfying $f_{1;k}(x)\le y\le f_{2;k}(x)$ necessarily lies
in $\Ac_{2}$. To this end fix a number $k\ge 1$ and consider all the points $(x,y)$ of the form
\begin{equation}
\label{eq:(x,y)=f2*(2x^2-1)}
(x,y) = (s,f_{2;k}(s)) \cdot (t,2t^{2}-1)
\end{equation}
with $s\in (0,\frac{1}{2k+1}]$, $t\in (0,1]$ (recalling the notation \eqref{eq:compwise prod} for componentwise multiplication).
Note that by the multiplicativity of $\Ac_{2}$ (Proposition \ref{prop:A monoid}) all the points
of the form \eqref{eq:(x,y)=f2*(2x^2-1)} are attainable $(x,y)\in \Ac_{2}$. Since
$f_{2;k}\left(\frac{1}{2k+1}\right)=1$, for $s=\frac{1}{2k+1}$ fixed, $t$ varying in $(0,1]$,
$(x,y)$ attains all the curve $(x,y)=(x,f_{1;k}(x))$; for $t=1$ fixed, $s$ varying in
$(0,\frac{1}{2k+1})$, $(x,y)$ attains the curve $(x,y)=(x,f_{2;k}(x)$.

We claim that for every $(x,y)$ with $f_{1;k}(x)\le y \le f_{2;k}(x)$ there exists $s,t$ in the range as above,
satisfying \eqref{eq:(x,y)=f2*(2x^2-1)}.
To show the latter statement, given such a point $(x,y)$ consider $s\in [x,\frac{1}{2k+1}]$ and $t=\frac{x}{s}$.
We are then to solve the equation $$y=f_{2;k}(s)\cdot \left(\frac{2x^{2}}{s^{2}}-1 \right)  $$ for the given
$y$, $s\in [\frac{1}{2k+1},1]$; as the r.h.s. of the latter equation attains the values $f_{1;k}(x)$
and $f_{2;k}(x)$ for $s=\frac{1}{2k+1}$ and $s=1$ respectively, we are guaranteed a solution by
the intermediate value theorem.
Geometrically, the above argument shows that as $s$ varies, the family
of parabolas
$$t\mapsto (s,f_{2;k}(s)) \cdot (t,2t^{2}-1)$$ tesselate the domain $\Dc_{0,x_{k}}(f_{1;k},f_{2;k})$
(cf. the proof of Proposition \ref{prop:x>1/3 under max attainable}
in section \ref{sec:x>1/3 under max attainable}).

\end{proof}

\subsection{Proof of Proposition \ref{prop:f2 finite} by convexity}

The convexity of the component-wise logarithm of a curve implies that finite products of points lying on that curve would
stay below it. We aim at eventually proving that all the curves $\gamma_{A}=(G_{A}(t),G_{A}(2t))$, $A\ge 3$ odd,
$t\in \left[\frac{\pi}{2}-\frac{1}{2A},\frac{\pi}{2}\right]$, satisfy the above property
(see Lemma \ref{lem:GA curv convex} below).
We exploit their convexity in Lemma \ref{lem:max conv inside},
which, after taking logarithm, is equivalent to the statement of Proposition
\ref{prop:f2 finite} (see the proof of Proposition \ref{prop:f2
  finite} below); the latter
follow from  finite products of points on a curve, with the property above,
staying below that curve.

\begin{lem}
\label{lem:GA curv convex}
Let $\eta_{A}$ be the curve
\begin{equation*}
\eta_{A}(t) = (\log(A\cdot |G_{A}(t)|),\log(G_{A}(2t))),
\end{equation*}
$t\in \left(\frac{\pi}{2}-\frac{\pi}{2A},\frac{\pi}{2}\right]$ with $A\ge 3$ odd.
Then in the above domain of $t$ both components of $\eta_{A}=(\eta_{A;1},\eta_{A;2})$ are strictly decreasing,
and moreover $\eta_{A}$ may be re-parametrized as $(z,h_{A}(z))$ with
$h_{A}:(-\infty,0]\rightarrow \R$
convex analytic, increasing, and $h(0)=0$.
\end{lem}

The somewhat technical proof of Lemma \ref{lem:GA curv convex} is
postponed to
Appendix \ref{apx:GA curv convex}.

\begin{lem}
\label{lem:max conv inside}
Let $\{h_{i}:(-\infty,0]\rightarrow\R\}_{i\le K}$ be a finite
collection of continuous convex functions
such that for all $i\le K$ we have $h_{i}(0)=0$.
Define $h:(-\infty,0]\rightarrow\R$ by
\begin{equation}
\label{eq:h max def}
h(z)=\sup\limits_{z_{i}\le 0:\: \sum\limits_{i=1}^{K}z_{i}=z}
\left\{\sum\limits_{i=1}^{K}h_{i}(z_{i}) \right\}.
\end{equation}
Then for every $z\in (-\infty,0]$ there exists an index
$i_{0}=i_{0}(z)$ so that
$
h(z)=h_{i_{0}}(z).
$
\end{lem}

Before giving a proof for Lemma \ref{lem:max conv inside} we may finally give a proof
for Proposition \ref{prop:f2 finite}.

\begin{proof}[Proof of Proposition \ref{prop:f2 finite} assuming
  lemmas \ref{lem:GA curv convex} and \ref{lem:max conv inside}]

Let $A=2k+1\ge 3$ be odd, and \eqref{eq:A prod odd Ai} be an arbitrary factorization of $A$ into
integers $A_{i}\ge 3$. Consider the curves
$\{\eta_{A_{i}}(t):\: t\in [\frac{\pi}{2}-\frac{\pi}{2A_{i}},\frac{\pi}{2} ]  \}_{i\le K}$ as defined
in \eqref{eq:etaA def}. By Lemma \ref{lem:GA curv convex} all of the $\eta_{A_{i}}$
can be re-parametrized as $(z_{i},h_{A_{i}}(z_{i}))$
on $(-\infty,0]$, with $h_{i}$ convex analytic and $h(0)=0$.

Hence, by Lemma \ref{lem:max conv inside} for every $x\in (0,\frac{1}{A}]$ there exists $i_{0}=i_{0}(x)$,
so that
\begin{equation*}
h(z):=\sup\limits_{z_{i}\le 0:\: \sum\limits_{i=1}^{K}z_{i}=z} \left\{\sum\limits_{i=1}^{K}h_{A_{i}}(z_{i}) \right\}
= h_{i_{0}}(z),
\end{equation*}
Note that, after taking logarithms, maximizing $\prod\limits_{i=1}^{K}G_{A}(2t_{i})$ under the constraint
$(t_{i})_{i\le K} \in \mathcal{X}_{\{A_{i}\}}(x)$ with $\mathcal{X}_{\{A_{i}\}}(x)$
as in \eqref{eq:Xc def}, $0<x\le \frac{1}{A}$ is equivalent to maximizing
$$\sum\limits_{i=1}^{K}\log G_{A}(2t_{i}) = \sum\limits_{i=1}^{K}h_{A_{i}}(z_{i})$$
under the constraint $\sum\limits_{i=1}^{K}z_{i}=z$, where $z=\log{Ax} \in (-\infty, 0]$.
More formally, recalling the definition \eqref{eq:etaA def} of $\eta_{A_{i}}$
and $(z_{i},h_{A_{i}}(z_{i}))$ being a re-parametrization of $\eta_{A_{i}}$,
the function $h(z)$ defined as in \eqref{eq:h max def}, on noting that
$z = \log Ax$,
satisfies
\begin{equation}
\label{eq:h(Ax=log g(x)}
h(Ax) = \sup\limits_{(t_{i})_{i\le K}\in \mathcal{Y}_{\{A_{i}\}}(x)} \left\{\prod\limits_{i=1}^{K}G_{A_{i}}(2t_{i}) \right\},
\end{equation}
where
\begin{equation*}
\mathcal{Y}_{\{A_{i}\}}(x) = \left\{(t_{i})_{i\le K}:\: \forall i. t_{i}\in \left[\frac{\pi}{2}-\frac{\pi}{2A_{i}},\frac{\pi}{2}\right],
\,\sum\limits_{i=1}^{K}\log(A_{i}|G_{A_{i}}(t_{i})|) =\log(Ax)\right\}.
\end{equation*}
Since $\sum\limits_{i=1}^{K}\log(A_{i}G_{A_{i}}(t_{i})) =\log(Ax)$ is equivalent to $\sum\limits_{i=1}^{K}\log(G_{A_{i}}(t_{i})) =\log(x)$ via
\eqref{eq:A prod odd Ai}, we have $\mathcal{Y}_{\{A_{i}\}}(x) = \mathcal{X}_{\{A_{i}\}}(x)$ (as in \eqref{eq:Xc def}),
and hence \eqref{eq:h(Ax=log g(x)} is
\begin{equation*}
h(Ax)=\log g_{\{A_{i}\}}(x).
\end{equation*}

The latter equality together with Lemma \ref{lem:max conv inside} then imply that
we have
\begin{equation*}
h_{i_{0}}(Ax)=\log g_{\{A_{i}\}_{i\le K}}(x)
\end{equation*}
for some $i_{0}\le K$; since $h_{i_{0}}$ is a re-parametrization of $\eta_{A_{i_{0}}}$, this is equivalent to
\begin{equation*}
(\log(A_{i_{0}}G_{A_{i_{0}}}(t_{i_{0}})),\log(G_{A_{i_{0}}}(2t_{i_{0}}))) = (\log(Ax),\log g_{\{A_{i}\}_{i\le K}}(x))
\end{equation*}
for some $t_{i_{0}}\in [\frac{\pi}{2}-\frac{\pi}{2A_{i_{0}}},\frac{\pi}{2}]$, i.e.
\begin{equation*}
\left(\frac{A_{i_{0}}}{A}G_{A_{i_{0}}}(t_{i_{0}}),G_{A_{i_{0}}}(2t_{i_{0}})\right) = (x, g_{\{A_{i}\}_{i\le K}}(x)),
\end{equation*}
which is the first statement of the present proposition, at least for $x>0$. For $x=0$ it is sufficient to
notice that for all $i\le K$, $$(G_{A_{i}}(t), G_{A_{i}}(2t))|_{t=\frac{\pi}{2}-\frac{\pi}{2A_{i}}} = (0,0),$$
so that in particular $g_{\{A_{i}\}_{i\le K}}(x) = 0$, whatever $\{A_{i}\}_{i\le K}$ are.

To see that the map $x\mapsto i_{0}(x)$ is in fact piecewise constant on $[0,\frac{1}{A}]$
(with finitely many pieces), we note that it is readily shown that on $(0, \frac{1}{A}]$,
$g_{\{ A_{i}\}_{i\le K}}$
is a maximum of finitely many analytic curves (namely,
$\left(\frac{A_{i}}{A}|G_{A_{i}}(t)|,G_{A_{i}}(2t)\right)$), and vanishes at $0$, which happens
to lie on all of them. Since such a collection of analytic curves may only intersect in
finitely many points for $x\in [0, \frac{1}{A}]$, it follows that $i_{0}(x)$ is uniquely
determined as the maximum of these outside of finitely many points (that include $(0,0)$),
and $i_{0}$ is constant between any two such consecutive points.

\end{proof}

\begin{proof}[Proof of Lemma \ref{lem:max conv inside}]
It is easy to check that with the assumptions of the present lemma,
the function $H:(-\infty,0]^{K}\rightarrow\R$ defined by
\begin{equation*}
H(t_{1},\ldots, t_{K}) = \sum\limits_{i=1}^{K}h_{i}(t_{i})
\end{equation*}
is a convex function. Now fix $t<0$ and consider the
set
$$\Omega(t):=\left\{(t_{i})_{i\le K}:\:
  \sum\limits_{i=1}^{K}t_{i}=t, \text{ $t_{i} \leq 0$ for $1 \le i \le
  K$} \right\}\subseteq (-\infty,0]^{K};$$
$\Omega(t)$ is a compact convex domain, and it is evident that
\begin{equation*}
h(t)=\max\limits_{(t_{i})\in\Omega(t)}H(t_{1},\ldots, t_{k}).
\end{equation*}

Now, a convex function cannot attain a maximum in the interior of a convex domain
(all the local extrema of a convex function are necessarily minima). Hence there exists
an index $i_{1}\le K$ so that $$h(t)=\sum\limits_{i=1}^{K}h_{i}(t_{i})$$
for some $(t_{i})\in\Omega(t)$ with $t_{i_{1}}=0$, i.e. one of the elements
of $(t_{i})$ must vanish. By induction,
we  find that all but one element of $(t_{i})$ vanish, say $t_{i}=0$ for
$i\ne i_{0}$, whence $t_{i_{0}}=t$, and $h(t)=h_{i_{0}}(t)$, as $h_{i}(0)=0$ for $i\ne i_{0}$
by the assumptions of the present lemma.

\end{proof}



\section{Proof of Theorem \ref{thm:sqrfree attainable class}: square-free attainable measures}

\label{sec:sqrfree attainable class}

\begin{proof}

Recall that we de-symmetrized all the probability measures by an analogue of
\eqref{eq:nun desymmetrized mun}.
First we show that \eqref{eq:square free nec suff} holds for any square-free attainable measure;
as the first inequality in \eqref{eq:square free nec suff} holds for every probability measure
\eqref{eq:P2 Riesz} it only remains to show that every point $(x,y) = (\hat{\mu}(1), \hat{\mu}(2))$
corresponding to a square-free attainable $\mu$ satisfies \eqref{eq:y<=M(x)}.

By the definition of square-free attainable measures,
if $\mu$ is square-free attainable then $(x,y)$ is lying in the closure of
the set of finite products
\begin{equation}
\label{eq:xy=prod}
\begin{split}
(\tilde{x},\tilde{y})&=\left\{\prod\limits_{i=1}^{K}(\cos(\theta_{i}),\cos(2\theta_{i})):\: \theta_{i}\in [0,\pi] \right\}
\\&=\left\{\prod\limits_{i=1}^{K}(x_{i},y_{i}):\: x_{i}\in [-1,1] \right\},
\end{split}
\end{equation}
where for all $i\le K$, $y_{i} = 2x_{i}^{2}-1$.
Now if $\tilde{y} > 0$ and $y_{i_{0}} < 0$ for some $i_{0}\le K$, then
$(\tilde{x},\tilde{y})\in \Ac_{2}^{-}$ is a mixed sign attainable point, and
(upon recalling Notation \ref{not:B1 B2 def})
Lemma \ref{lem:mix sign in B1 or B2} implies that $(\tilde{x},\tilde{y}) \in B_{1}$, i.e.,
$|\tilde{x}| \leq 1/2$ and $\tilde{y} \leq (2|\tilde{x}|-1)^{2}$.

If $\tilde{y} > 0$ and $y_{i} \geq 0$ for all $i$, then
$y_{i} =2x_{i}^{2}-1 \le x_{i}^{4}$ for all $i$ as it is easy to check
the latter inequality explicitly, consequently $\tilde{y} \leq \tilde{x}^{4}$.
Since \eqref{eq:y<=M(x)} holds on the collection of all products \eqref{eq:xy=prod},
it also holds on its closure, namely for square-free attainable measures.
This concludes the proof of the necessity of the inequality
\eqref{eq:square free nec suff}.


It then remains to show the sufficiency, i.e. any point $(x,y)$ satisfying \eqref{eq:square free nec suff}
corresponds to a square-free attainable measure. We claim that the
attainable measures constructed by Proposition \ref{prop:x>1/3 under max attainable} are in fact
square-free attainable. To this end it is crucial to notice that
the measures corresponding to points lying on the curves $$\{(x,x^{4}):x\in [0,1]\}$$
(constructed by Lemma \ref{lem:x^4 attainable}), and $$\{(x,(2x-1)^{2}): x\in [0,1] \}$$ (a product of the parabola
$y=x^{2}$ by itself) exploited in the course of the proof of Proposition \ref{prop:x>1/3 under max attainable}
are both square-free attainable. We recall in addition, that collection of square-free attainable
measures is closed under convolutions, so that the products of points
corresponding to square-free
attainable measures correspond to square-free attainable measures;
hence  the tessellation
argument used in the proof of Proposition \ref{prop:x>1/3 under max attainable} works in this case too.

\end{proof}

\appendix

\section{Proof of Proposition \ref{prop:x<1/3 below mix cur}: below the ``mixed signs" curve $y=(2x-1)^{2}$}

\label{apx:below mixed big small}

By the assumptions of Proposition \ref{prop:x<1/3 below mix cur} there
exists $i$ such that $t_{i} \in [\pi/(2A_{i}), \pi/2-\pi/(2A_{i})]$
(for $A_{i}$ odd), or $t_{i} \in [\pi/(2A_{i}), \pi/2]$ (for $A_{i}$ even.)
%
%
The following lemma exploits this property to yield
more information about (at least) one point in the product.

\begin{lem}
\label{lem:below mix cur single}


Let $A\ge 3$ and $(x,y)=(G_{A}(t),G_{A}(2t))$. If $A$ is odd
and $t \in [\frac{\pi}{2A},\frac{\pi}{2}-\frac{\pi}{2A}]$, or $A$ is even
and $t \in [\frac{\pi}{2A},\frac{\pi}{2}]$, then either $y\le 0$, or $y\le (2|x|-1)^{2}$ and $|x|<\frac{1}{3}$.

 If $A=2$ and $t \in \left[\frac{\pi}{4},\frac{\pi}{2}\right]$, then $y=G_{2}(2t) \le 0$.


\end{lem}

\begin{proof}[Proof of Proposition \ref{prop:x<1/3 below mix cur} assuming Lemma
\ref{lem:below mix cur single}]

Assume with no loss of generality that the postulated index
is $i=1$, i.e.
$(x_{1},y_{1})=(G_{A_{1}}(t_{1}),G_{A_{1}}(2t_{1}))$ with either
$A_{1} \ge 3$ being odd and $t \in [\frac{\pi}{2A_{1}},\frac{\pi}{2}-\frac{\pi}{2A_{1}}]$, or
$A_{1}\ge 2$ being even and $t \in
[\frac{\pi}{2A_{1}},\frac{\pi}{2}]$. Suppose first that
$y_{1}\le 0$. In this case the point $(x,y)$ is ``mixed sign attainable" (cf. Definition
\ref{def:tot pos not}), so that Lemma \ref{lem:not tot pos under max curve} implies that $y\le (2|x|-1)^{2}$.

Otherwise we assume that $y_{1}>0$ and $y>0$. Then Lemma \ref{lem:below mix cur single} implies that
$A\ge 3$, and $|x_{1}|<\frac{1}{3}$, whence
\begin{equation*}
0<y\le y_{1} \le (2|x_{1}|-1)^{2} \le (2|x|-1)^{2},
\end{equation*}
since $|x|\le |x_{1}|$ and the function $x\mapsto (2x-1)^{2}$ is decreasing on $\left[0,\frac{1}{2}\right]$.

\end{proof}

\begin{proof}[Proof of Lemma \ref{lem:below mix cur single}]

First, upon recalling that for $A=2$ we have $G_{2}(t)=\cos(t)$, the second statement of Lemma \ref{lem:below mix cur single} is obvious.
We are left with proving the first statement. For $A=3$ if $t \in \left[\frac{\pi}{6},\frac{\pi}{3}\right]$,
then $$y=\frac{\sin(6t)}{3\sin(2t)} \le 0$$ again.
We  may thus assume that $A\ge 4$.

Next, we would like to consolidate the even and the odd $A$ cases,
by showing  that if $A$ is
even and $t\in \left[\frac{\pi}{2}-\frac{\pi}{2A}, \frac{\pi}{2}\right]$,
then the statement of the present lemma holds.
To do this we note that in this range $2At\in [(A-1)\pi,A\pi]$,
so that
\begin{equation*}
G_{A}(2t)=\frac{\sin(2At)}{A\sin(2t)}\le 0
\end{equation*}
once more.

Hence we may assume that $t\in \left[\frac{\pi}{2A}, \frac{\pi}{2}-
\frac{\pi}{2A} \right]$, whether $A$ is even or odd. We would like to further cut out
the short intervals $\left[\frac{\pi}{2A}, \frac{\pi}{A}\right]$ and
$\left[\frac{\pi}{2}-\frac{\pi}{A}, \frac{\pi}{2}-\frac{\pi}{2A}\right]$, i.e.
establish the validity of the present lemma in these intervals.
If $t\in \left[\frac{\pi}{2A}, \frac{\pi}{A}\right]$
whether $A$ is even or odd, then $2At\in [\pi,2\pi]$, so that $y=G_{A}(2t) \le 0$ in this regime too.

If $t\in \left[\frac{\pi}{2}-\frac{\pi}{A}, \frac{\pi}{2}-\frac{\pi}{2A}\right]$, then
$2At\in [(A-2)\pi,(A-1)\pi]$, so that if $A$ is odd then
$y=G_{A}(2t)=\frac{\sin(2At)}{A\sin(2t)}\le 0$.
In the remaining case $A$ even, for the same range $t\in
\left[\frac{\pi}{2}-\frac{\pi}{A},
  \frac{\pi}{2}-\frac{\pi}{2A}\right]$, we write $A=2B$ for $B\in\Z$,
and note that
\begin{equation*}
\begin{split}
(x,y)&=(G_{A}(t),G_{A}(2t)) =
\left(\frac{\sin(Bt)\cos(Bt)}{B\sin(t)},\frac{\sin(2Bt)\cos(2Bt)}{B\sin(2t)}\right)
\\&= (G_{B}(t),G_{B}(2t))\cdot (G_{2}(t),G_{2}(2t)).
\end{split}
\end{equation*}
Hence if in  turn $B$ is even, then $G_{B}(2t) =
\frac{\sin(2Bt)}{B\sin(2t)}\le 0$, since
$2Bt\in [(B-1)\pi,(B-1)\pi+\frac{\pi}{2}].$ Hence $(x,y)$ is mixed sign attainable,
and therefore by Lemma \ref{lem:not tot pos under max curve}, $y\le (2|x|-1)^{2}$, and, in addition, $|x|\le \frac{1}{3}$ by
Lemma \ref{lem:x-bigger-one-third}.

Otherwise, if $B$ is odd, we may assume that $A\ge 6$ is even
(in the same range $t\in \left[\frac{\pi}{2}-\frac{\pi}{A}, \frac{\pi}{2}-\frac{\pi}{2A}\right]$);
in this case we claim that $|x|=|G_{A}(t)| \le \frac{1}{5}$ and $y= |G_{A}(2t)| \le \frac{1}{3}$.
As $\frac{1}{3}\le (2/5-1)^{2}$, and $x\mapsto (2x-1)^{2}$ is decreasing on $[0,\frac{1}{2}]$
this is sufficient to yield $y\le (2|x|-1)^{2}$. To show this, we
first note that $G_{A}(2t)=\pm G_{A}(2(\pi/2-t))$; hence
Lemma \ref{lem:x-bigger-one-third} implies that $y\le \frac{1}{3}$ indeed. Concerning the value of $|x|$, we have
for $t$ in the range as above (bearing in mind that $A\ge 6$):
\begin{equation*}
\begin{split}
|G_{A}(t)| &\le \frac{1}{A\sin(t)}\le \frac{1}{A\sin(\pi/2-\pi/A)} = \frac{1}{A\cos(\pi/A)} \\&\le
\frac{1}{6\cos(\pi/6)} = 0.19\ldots < \frac{1}{5},
\end{split}
\end{equation*}
since $A\mapsto A\cdot \cos(\pi/A)$ is strictly increasing for $A \ge 6$.

Finally, we take care of the case $A\ge 4$, whether $A$ is even or
odd, and the remaining range
\begin{equation}
\label{eq:t in [pi/A,pi/2-pi/A}
t\in \left[\frac{\pi}{A},\frac{\pi}{2}-\frac{\pi}{A}\right],
\end{equation}
and $(x,y)=(G_{A}(t),G_{A}(2t))$.
Noting that $\sin(t)\ge \frac{2}{\pi}t$ everywhere on
$[0,\frac{\pi}{2}]$, we find that
for $t\in [\frac{2\pi}{A},\frac{\pi}{2}]$,
\begin{equation*}
|G_{A}(t)|\le \frac{1}{A\sin(t)} \le \frac{\pi}{2}\frac{1}{A\cdot 2\pi/A} = \frac{1}{4}.
\end{equation*}
Hence (under the assumption \eqref{eq:t in [pi/A,pi/2-pi/A} on $t$), if $t>\frac{2\pi}{A}$,
$|x|=|G_{A}(t)| \le \frac{1}{4}$, and (using the natural symmetry $G_{A}(t)=\pm G_{A}(\pi-t)$),
$y\le |y| \le G_{A}(2t) \le \frac{1}{4}$.

If both $|x|\le \frac{1}{4}$ and $y\le \frac{1}{4}$, then $y\le (2|x|-1)^{2}$, as $x\mapsto (2x-1)^{2}$
is decreasing on $[0,\frac{1}{2}]$. Hence we are left with taking care of the range
$t\in [\frac{\pi}{A},\frac{2\pi}{A}]$, where we still have $y\le \frac{1}{4}$, and we may assume $x>\frac{1}{4}$.
Moreover, if $t\in [\frac{3\pi}{2A},\frac{2\pi}{A}]$, $2At\in
[3\pi,4\pi]$, so that $y=G_{A}(2t) \le 0$, hence
it is enough  to prove the statement for $t\in
[\frac{\pi}{A},\frac{3\pi}{2A}]$.

Now, recall that by Lemma \ref{lem:sinc-decreasing} the function $t\mapsto \frac{\sin{t}}{t}$ is decreasing on $[0,\pi]$, so that,
bearing in mind that $A\ge 4$,
\begin{equation*}
\frac{\sin{t}}{t}\ge \frac{\sin(At/4)}{At/4},
\end{equation*}
and thus
\begin{equation}
\label{eq:x<=x'}
\begin{split}
|x|&=|G_{A}(t)| = \frac{|\sin(At)|/(At)}{|\sin(t)|/t} \le
\frac{|\sin(At)|/(At)}{\sin(At/4)/(At/4)} \\&= \frac{|\sin(At)|}{4\sin(At/4)} = |G_{4}(s)| =:|x'|,
\end{split}
\end{equation}
where we rescale by letting $s = \frac{At}{4} \in
[\frac{\pi}{4},\frac{3\pi}{8}]$.
Arguing along the same lines we obtain
\begin{equation}
\label{eq:y<=y'}
|y|=|G_{A}(2t)|\le |G_{4}(2s)| =: |y'|
\end{equation}
(note that $2At/4=At/2 <\pi$, so that Lemma \ref{lem:sinc-decreasing} is valid in this range).

Since $$G_{4}(s) = \frac{\sin(4s)}{4\sin(s)} = \cos(s)\cos(2s)=G_{2}(s)\cdot G_{2}(2s),$$
we have that
\begin{equation*}
(x',y')=(G_{4}(s),G_{4}(2s)) = (G_{2}(s),G_{2}(2s))\cdot (G_{2}(2s),G_{2}(4s)),
\end{equation*}
is a product of two attainable points, and moreover, since $s\in
[\frac{\pi}{4},\frac{3\pi}{8}]$,
$G_{2}(2s) =\cos(2s) \le 0$ (and also $G_{2}(4s)\le 4$). That means that $(x',y')$ is
``mixed sign attainable'' (cf. Definition \ref{def:tot pos not}), and hence Lemma \ref{lem:not tot pos under max curve}
implies that $y'\le (2|x'|-1)^2$. Finally, bearing in mind \eqref{eq:x<=x'} and \eqref{eq:y<=y'},
as well as $x\mapsto (2x-1)^2$ decreasing on $[0,\frac{1}{2}]$, we have
\begin{equation*}
y\le |y'| \le (2x'-1)^{2} \le (2x-1)^{2}.
\end{equation*}

\end{proof}

\section{Proof of auxiliary technical lemmas}

\label{apx:GA curv convex}

\begin{proof}[Proof of Lemma \ref{lem:GA curv convex}]

  First, by using some simple trigonometric identities (in particular,
  that $\sin(\pi/2-t) = \cos( t)$), we may re-parametrize
  $\eta_{A}(t)$ as
$$
\eta_{A}(t) = (x(t),y(t))=\left( \log
  \left(A\frac{\cos(At)}{A\cos(t)}\right), \
  \log\left(\frac{\cos(At)}{\cos(t)}\cdot
\frac{\sin(At)}{A\sin(t)}\right)  \right)
$$
\begin{multline*}
=\bigg( \log \cos(At)-\log(\cos(t)), \\
\log(\cos(At))-\log(\cos(t)) +\log(\sin(At))-\log(A\sin(t)  \bigg),
\end{multline*}
for $t\in [0,\frac{\pi}{2A}]$. By taking the derivatives, it is easy
to see that both $x(t)$ and $y(t)$
are strictly decreasing, thus, by the implicit function theorem,
the curve $(x(t),y(t))$ can be re-parametrized as $(x,h_{A}(x))$ with
$h_{A}:(-\infty,0]\rightarrow \R$ analytic and strictly increasing.
Hence to prove that $\eta_{A}$ is convex (or equivalently, that
$h_{A}$ is convex),
it is sufficient to show that the slope
$$\frac{dy}{dx}=\frac{y'(t)}{x'(t)} = 1+\frac{(\log(\sin(At))-\log(A\sin{t}))'}
{(\log(\cos(At))-\log(\cos{t}))'}$$ is decreasing on
$(0,\frac{\pi}{2A})$, which in turn is equivalent to the
function $$t\mapsto \frac{(\log(\sin(At))-\log(\sin{t}))'}
{(\log(\cos(At))-\log(\cos{t}))'}$$ being decreasing on the same
domain. We rescale by setting $s=At$ and let $\alpha:=\frac{1}{A}\in
(0,\frac{1}{3}]$, $g(s):=-\log(\sin(s))$, $f(s):=-\log(\cos(s))$; we
are then to prove that
\begin{equation*}
s\mapsto \frac{(g(s)-g(\alpha s))'}{(f(s)-f(\alpha s))'}
\end{equation*}
is decreasing on $(0,\frac{\pi}{2})$.

Recall the product expansion formulas
\begin{equation*}
\sin(x) = x\prod\limits_{k=1}^{\infty}
\left(1-\frac{x^{2}}{k^{2}\pi^{2}} \right),
\quad
\cos(x) = \prod\limits_{k=1}^{\infty} \left(1-\frac{4x^{2}}{(2k-1)^{2}\pi^{2}} \right)
\end{equation*}
of the sine and cosine respectively, and the Taylor series expansion
$-\log(1-x)=\sum\limits_{k=1}^{\infty}\frac{x^{k}}{k}.$ With the notation as above we then have
\begin{equation*}
f(s) = \sum\limits_{i=1}^{\infty} a_{i}s^{2i},
\quad \quad
g(s)+\log(s) = \sum\limits_{j=1}^{\infty} b_{j}s^{2j},
\end{equation*}
with
\begin{equation*}
a_{i} = \frac{2^{2i}\zeta^{*}(2i)}{i\pi^{2i}}>0  ; \quad \; b_{j}=\frac{\zeta(2j)}{j\pi^{2j}}>0,
\end{equation*}
where $\zeta$ is the usual Riemann Zeta function, and
$\zeta^{*}(s):= \sum_{k=1}^{\infty}\frac{1}{(2k-1)^{s}},$ for
$ s>1$.

We then have
\begin{equation*}
F(s):= f(s)-f(\alpha s) = \sum\limits_{i=1}^{\infty}a_{i}(1-\alpha^{2i})s^{2i},
\end{equation*}
and
\begin{equation*}
\begin{split}
G(s)&:= g(s)+\log(s)-(g(\alpha s)+\log(\alpha s))
=g(s)-g(\alpha s)-\log(\alpha) \\&= \sum\limits_{j=1}^{\infty}b_{j}(1-\alpha^{2j})s^{2j}-\log(\alpha),
\end{split}
\end{equation*}
and we need to prove that $$G''(s)F'(s)-G'(s)F''(s)<0$$ on $s\in (0,\frac{\pi}{2})$; note that the latter is defined
and analytic on the  interval $(0,\frac{\pi}{2})$.
Now, we have
\begin{equation*}
\begin{split}
G''(s)F'(s) &= \sum\limits_{j=1}^{\infty}b_{j}\cdot 2j(2j-1)(1-\alpha^{2j})s^{2j-2}\cdot
\sum\limits_{i=1}^{\infty}a_{i}\cdot
2i(1-\alpha^{2i})s^{2i-1}\\&=4a_{1}b_{1}s+\sum\limits_{k=1}^{\infty}c_{k}s^{2k+1},
\end{split}
\end{equation*}
and
\begin{equation*}
\begin{split}
G'(s)F''(s) &= \sum\limits_{j=1}^{\infty}b_{j}\cdot 2j(1-\alpha^{2j})s^{2j-1}\cdot
\sum\limits_{i=1}^{\infty}a_{i}\cdot
2i(2i-1)(1-\alpha^{2i})s^{2i-2}\\
&=4a_{1}b_{1}s+\sum\limits_{k=1}^{\infty}d_{k}s^{2k+1},
\end{split}
\end{equation*}
and similarly
\begin{equation*}
g''(s)f'(s)=\frac{1}{3}s+\sum\limits_{k=2}^{\infty}\gamma_{k}s^{2k+1}
\end{equation*}
and
\begin{equation*}
g'(s)f''(s)=\frac{1}{3}s+\sum\limits_{k=2}^{\infty}\delta_{k}s^{2k+1},
\end{equation*}
where for $k\ge 2$ we have $0<c_{k}<\gamma_{k}$, and (since $a_{i},
b_{j} \geq 0$ together with $\alpha \leq 1/3$)
$$d_{k}\ge (1-\alpha^{2})(1-\alpha^{4}) \delta_{k}>\frac{3}{4}\delta_{k}>0.$$

Hence
\begin{equation}
\label{eq:G''F'g''f'}
G''(s)F'(s)-4a_{1}b_{1}s < g''(s)f'(s)-\frac{1}{3}s
\end{equation}
and
\begin{equation}
\label{eq:G'F''g'f''}
G'(s)F''(s)-4a_{1}b_{1}s > \frac{3}{4}\left(g'(s)f''(s)-\frac{1}{3}s\right).
\end{equation}
In a moment we are going to show that the inequality
\begin{equation}
\label{eq:g''f'-g'f''>2}
\frac{g'(s)f''(s)-\frac{1}{3}s}{g''(s)f'(s)-\frac{1}{3}s} \ge 2
\end{equation}
holds for $s\in\frac{\pi}{2}$. Assuming \eqref{eq:g''f'-g'f''>2}, use \eqref{eq:G''F'g''f'}
and \eqref{eq:G'F''g'f''} to finally obtain (note that $\gamma_{k} >
0$ for all $k$)
\begin{equation*}
\begin{split}
G''(s)F'(s) - G'(s)F''(s) &< \left(g''(s)f'(s)-\frac{1}{3}s\right) -
\frac{3}{4}\left(g'(s)f''(s)-\frac{1}{3}s\right)\\&
< -\frac{1}{2}\left(g''(s)f'(s)-\frac{1}{3}s\right)<0.
\end{split}
\end{equation*}

To see \eqref{eq:g''f'-g'f''>2} we note that the involved ratio equals to precisely $2$
at $s=0$, and claim that $$  \frac{d}{ds}K(s):=\frac{d}{ds}\left[\frac{g'(s)f''(s)-\frac{1}{3}s}{g''(s)f'(s)-\frac{1}{3}s} \right]>0$$
for $s\in [0,\frac{\pi}{2}]$. The latter derivative equals
\begin{equation*}
\frac{d}{ds}K(s)= -\frac{q(s)}{\cos(s)^{2}\left( s^{2}-\sin(s)^{2}  \right)^{2}},
\end{equation*}
where $q(s)$ is given by
\begin{equation}
\label{eq:boring h def}
\begin{split}
q(s) &= 2\cos(s)^{3}\sin(s)s^2-\cos(s)^3\sin(s)+\cos(s)\sin(s)s^2\\&-
4\cos(s)^2\sin(s)^2s-s^3+\sin(s)\cos(s)+s\sin(s)^2.
\end{split}
\end{equation}
Thereupon the inequality \eqref{eq:g''f'-g'f''>2} finally follows from
Lemma \ref{lem:h(s) tedious <=0} below.
\end{proof}

\begin{lem}
\label{lem:h(s) tedious <=0}
The function $q(s)$, defined by \eqref{eq:boring h def}, satisfies
$q(s)\le 0$ on $s\in \left[0,\frac{\pi}{2}\right]$.
\end{lem}

\begin{proof}
  We remark that the lemma is evident from plotting $q(s)$
  numerically, but a formal argument can be given along the following
  lines.  We Taylor expand $q$
  around $s=0$ (we caution the reader that $d_{k}$ is not the same as
  in the proof of the previous Lemma):
\begin{equation}
\label{eq:h(s) Taylor}
q(s) = \sum\limits_{k=4}^{\infty}d_{k}s^{2k+1},
\end{equation}
where
\begin{equation*}
d_{k}=(-1)^{k+1} \left(\frac{2^{2k-1}+2^{4k-4}}{(2k-1)!}+\frac{2^{2k-1}-2^{4k-1}}{(2k)!} +\frac{2^{4k-1}-2^{2k-1}}{(2k+1)!}\right);
\end{equation*}
in particular $d_{4}=-\frac{16}{135}$, $d_{5}=\frac{16}{315}$, $d_{6}=-\frac{16}{1575}$, $d_{7}=\frac{2944}{2338875}$.
The general formula clearly implies that as $k\rightarrow \infty$,
$d_{k}\sim (-1)^{k} \frac{2^{4k-4}}{(2k-1)!}, $ and moreover, a  crude
estimate shows that
$$d_{k} = (-1)^{k} \frac{2^{4k-4}}{(2k-1)!}\left(1+\theta\left(\frac{1}{2^{2k-3}}+\frac{4}{k}+\frac{1}{k\cdot 2^{2k-2}}+
    \frac{4}{k^{2}} \right) \right),$$ where\footnote{In writing this
  way we follow Vinogradov: the exact value of $\theta$ might
  change, but the inequality $|\theta| \leq 1$ always holds.} $|\theta
|\le 1$. For
$k\ge 8$ we then
have
\begin{equation}
\label{eq:dk effective asymptotics}
d_{k} = (-1)^{k} \frac{2^{4k-4}}{(2k-1)!}\left(1+\frac{5}{8}\theta\right);
\end{equation}
it is evident that the signs of $d_{k}$ are alternating.

Now separate the summands of \eqref{eq:h(s) Taylor} corresponding to $k\le 7$ from the rest;
the remaining summands are united into pairs, i.e. write
\begin{equation}
\label{eq:h taylor unite}
q(s) = s^{9}q_{0}(s)+\sum\limits_{r=4}^{\infty}\left(d_{4r+1}s^{4r+1}+d_{4r+3}s^{4r+3}\right),
\end{equation}
where $$q_{0}(s)=\sum\limits_{k=4}^{7}d_{k}s^{2k+1} = -\frac{16}{135}+
\frac{16}{315}s^{2}-\frac{16}{1575}s^{4}+\frac{2944}{2338875}s^{6},$$
using the explicit Taylor coefficients mentioned above. First, it is
tedious but straightforward
to see that $q_{0}(s) \le 0$ on $s\in
[0,\frac{\pi}{2}]$.

For the remaining terms, note
that by the above, for $r\ge 4$ we have $d_{4r+1}<0$ and $d_{4r+3}>0$, and
upon employing \eqref{eq:dk effective asymptotics} twice, we
obtain (note that since $r \geq 4$ we have $(4r+2) \geq 18 > 2^{4}$
and thus $(8r+5) \cdots (4r+2) > 2^{16r}$)
\begin{equation*}
\begin{split}
|d_{4r+3}| &< \frac{13}{8}\frac{2^{8r}}{(4r+1)!} < \frac{13}{8}\cdot 16 \cdot
\frac{1}{(4r)^{2}}\frac{2^{8r-4}}{(4r-1)!} \\&\le \frac{13}{8} \cdot \frac{1}{r^{2}}\cdot \frac{8}{3} |d_{4r+1}|
= \frac{13}{3r^{2}}|d_{4r+1}|< 0.3|d_{4r+1}|.
\end{split}
\end{equation*}
Hence each of the summands in \eqref{eq:h taylor unite}, for
$s\in [0,\frac{\pi}{2}]$, satisfies:
\begin{equation*}
d_{4r+1}s^{4r+1}+d_{4r+3}s^{4r+3}< d_{4r+1}s^{4r+1}+0.3 \left(\frac{\pi}{2}\right)^{2}|d_{4r+1}|s^{4r+1} < 0,
\end{equation*}
as $0.3 \left(\frac{\pi}{2}\right)^{2} < 1.$ Finally $q(s)<0$, since
all the summands in \eqref{eq:h taylor unite}
are negative.

\end{proof}

\begin{proof}[Proof of Lemma \ref{lem:GA analyt par}]

By Lemma~\ref{lem:GA curv convex}
(note that the proof of Lemma~\ref{lem:GA curv convex} does not use
Lemma~\ref{lem:GA analyt par})
we may re-parametrize
$\eta_{A}$ as
$(x,h_{A}(x))$ on $x\in (-\infty,0]$.
Since both components $\eta_{A;1}$ and $\eta_{A;2}$ are strictly decreasing, it follows
that $h_{A}'(x)>0$ everywhere, and $h_{A}'(x)\le \frac{4}{3}$ follows from the convexity of
$h_{A}$, and the explicit computation $h_{A}'(0)=\frac{4}{3}$.

\end{proof}

\begin{proof}[Proof of Corollary \ref{cor:y^4/3>=1/Ax}]

By the multiplicativity, it is sufficient to prove the statement for a single $A_{i}$, i.e.
that if $$(x,y)=(G_{A}(t),G_{A}(2t))$$ with $A$ odd and $t\in [\frac{\pi}{2}-\frac{\pi}{2A},\frac{\pi}{2}]$,
then $$y\ge (Ax)^{4/3}.$$ As we may assume with no loss of generality
that $x>0$ (note that $y>0$ by the assumption of $t_{i}$ being near $\pi/2$)
the latter is equivalent to
\begin{equation}
\label{eq:logy>=4/3log(Ax)}
\log y \ge \frac{4}{3} \log(Ax).
\end{equation}
Note that, with $\eta_{A}$ defined as in Lemma \ref{lem:GA analyt par},
$\eta_{A}(t)= (z,h_{A}(z))=(\log(Ax),\log(y))$, with $h_{A}$ analytic convex,
$h_{A}(0)=0$, and a straightforward computation shows that
$h_{A}'(0)=\frac{4}{3}$.
By the convexity of $\eta_{A}$ then the curve
lies above its tangent line at the origin,
i.e. \eqref{eq:logy>=4/3log(Ax)} follows.

\end{proof}

\begin{proof}[Proof of Lemma \ref{lem:convexity log(2exp(2z)-1)}]
The claimed inequality follows from the identity
\begin{equation*}
(2x_{1}^{2}-1)(2x_{2}^{2}-1)-(2(x_{1}x_{2})^{2}-1) = 2(x_{1}^{2}-1)(x_{2}^{2}-1).
\end{equation*}
\end{proof}


\end{document}